\documentclass{amsart}
\usepackage{graphicx,mathrsfs,setspace,xypic,fancyhdr}
\usepackage{latexsym}
\usepackage{amscd, amsfonts, eucal, mathrsfs, amsmath, amssymb, amsthm}
\newtheorem{thm}{Theorem}[section]
\newtheorem{cor}[thm]{Corollary}
\newtheorem{lem}[thm]{Lemma}
\newtheorem{prop}[thm]{Proposition}
\theoremstyle{definition}
\newtheorem{defn}[thm]{Definition}
\theoremstyle{remark}
\newtheorem{rem}[thm]{Remark}
\numberwithin{equation}{section}

\begin{document}

\title{ The geometry of Ulrich bundles on del Pezzo surfaces}
\author{Emre Coskun}
\address{Department of Mathematics, University of Western Ontario, London, ON, N6A 5B7 CANADA}
\email{ecoskun@uwo.ca}
\author{Rajesh S. Kulkarni}
\address{Department of Mathematics, Michigan State University, East Lansing, MI 48824}
\email{kulkarni@math.msu.edu}
\author{Yusuf Mustopa}
\address{Department of Mathematics, University of Michigan, Ann Arbor, MI 48109}
\email{ymustopa@umich.edu}


\date{\today}
\begin{abstract}
Given a smooth del Pezzo surface $X_d \subseteq \mathbb{P}^{d}$ of degree $d,$ we show that a smooth irreducible curve $C$ on $X_d$ represents the first Chern class of an Ulrich bundle on $X_d$ if and only if its  kernel bundle $M_C$ admits a generalized theta-divisor.  

This result is applied to produce new examples of complete intersection curves with semistable kernel bundle, and also combined with work of Farkas-Musta\c{t}\v{a}-Popa to relate the existence of Ulrich bundles on $X_d$ to the Minimal Resolution Conjecture for curves lying on $X_d.$  In particular, we show that a smooth irreducible curve $C$ of degree $3r$ lying on a smooth cubic surface $X_3$ represents the first Chern class of an Ulrich bundle on $X_3$ if and only if the Minimal Resolution Conjecture holds for $C.$    
\end{abstract}
\maketitle
\thispagestyle{plain}

\section{Introduction}

Let $X \subseteq \mathbb{P}^{n}$ be a smooth projective variety of degree $d$ and dimension $k.$  Recall that a vector bundle $\mathcal{E}$ on $X$ is ACM (Arithmetically Cohen-Macaulay) if $\Gamma_{\ast}(\mathcal{E}):=\oplus_{m \in \mathbb{Z}}H^0(\mathcal{E}(m))$ is a graded Cohen-Macaulay module over $S:=\textnormal{Sym}^{\ast}H^{0}(\mathcal{O}_{\mathbb{P}^{N}}(1)).$  These bundles, which have been studied by numerous authors (see \cite{Fae} and the references therein), occur naturally in areas such as liaison theory and the study of singularities.  

In this article we are concerned with the ``nicest" ACM bundles on $X,$ a notion which we now make precise.  If $M$ is a graded Cohen-Macaulay $S-$module supported on $X,$ then by the Auslander-Buchsbaum theorem, $M$ admits a graded minimal resolution of the form
\begin{equation}
\label{acmres}
0 \leftarrow M \leftarrow \oplus_{i}S(-a_{0,i})^{\beta_{0,i}} \xleftarrow{\phi_{0}} \cdots \xleftarrow{\phi_{n-k-1}} \oplus_{i}S(-a_{n-k,i})^{\beta_{n-k,i}} \leftarrow 0
\end{equation}
We say that the $S-$module $M$ is \textit{Ulrich} if $a_{j,i} = j$ for $0 \leq j \leq n-k$ and all $i.$  This notion grew out of work of Ulrich on Gorenstein rings \cite{Ulr} and was studied extensively in the papers \cite{BHS, BHU1, BHU2}.  Up to twisting, Ulrich $S-$modules (which have also been referred to as ``linear maximal Cohen-Macaulay modules") are precisely the graded Cohen-Macaulay modules for which the maps $\phi_{i}$ are all matrices of linear forms.  

Accordingly, a vector bundle $\mathcal{E}$ on $X$ is said to be an \textit{Ulrich bundle} if $\Gamma_{\ast}(\mathcal{E})$ is an Ulrich $S-$module.  The body of work on Ulrich modules cited in the previous paragraph implies that Ulrich bundles exist on curves, linear determinantal varieties, hypersurfaces, and complete intersections.  Moreover, Ulrich bundles are semistable in the sense of Gieseker (e.g. Proposition \ref{giesstab}), so once we fix rank and Chern class they may be parametrized up to $S-$equivalence by a quasi-projective scheme.  

In the case where $X$ is a hypersurface in $\mathbb{P}^{n}$ defined by a homogeneous form $f(x_0, \cdots ,x_n),$ Ulrich bundles on $X$ correspond to linear determinantal descriptions of powers of $f$ (e.g. \cite{Bea}). This has been generalized to the case of arbitrary codimension in \cite{ESW}, where Ulrich bundles are shown to correspond to linear determinantal descriptions of powers of Chow forms.  

Around the same time as the appearance of \cite{Ulr}, van den Bergh effectively showed in \cite{vdB} that when $X \subseteq \mathbb{P}^{n}$ is a hypersurface of degree $d$ defined by the equation $w^d=g(x_1, \cdots ,x_n),$ Ulrich bundles on $X$ correspond to representations of the generalized Clifford algebra of $g.$  For recent work based on this correspondence, we refer to \cite{Cos, Kul, CKM1, CKM2}. 

It is clear by now that the concept of an Ulrich bundle encodes substantial algebraic information.  Our goal is to show that (at least) in the case of a del Pezzo surface, it also encodes substantial geometry.  

To motivate the source of this geometry, we first consider Ulrich line bundles on a smooth irreducible curve $X \subseteq \mathbb{P}^{n}$ of degree $d$ and genus $g.$  We may associate to the line bundle $\mathcal{O}_{X}(-1)$ a theta-divisor in $\textnormal{Pic}^{g-1+d}(X),$ which is defined set-theoretically as 
\begin{equation}
\Theta_{\mathcal{O}_{X}(-1)}:=\{ L \in \textnormal{Pic}^{g-1+d}(C) : H^0(L(-1)) \neq 0\}
\end{equation}
Since this is a translate of the classical theta-divisor on $\textnormal{Pic}^{g-1}(X),$ it is an ample effective divisor on $\textnormal{Pic}^{g-1+d}(X).$  

Using (v) of Proposition \ref{cliffdef}, one sees that a line bundle $L$ on $X$ is Ulrich precisely when it is of degree $g-1+d$ and its isomorphism class lies in the Zariski-open subset $\textnormal{Pic}^{g-1+d}(X)  \setminus  \Theta_{\mathcal{O}_{X}(-1)},$ i.e. it satisfies the vanishings 
\begin{equation}
H^0(L(-1)) = H^1(L(-1)) =0.
\end{equation}  
In particular, Ulrich line bundles on $X$ are parametrized up to isomorphism by an affine variety.  This characterization generalizes naturally to Ulrich bundles of rank $r \geq 2$ on $X$, if we replace $\textnormal{Pic}^{g-1+d}(X)$ with the moduli space $U_X(r,r(g-1+d))$ of semistable vector bundles on $X$ of rank $r$ and degree $r(g-1+d).$  

Turning to higher dimensions, one has a statement of this type for rank-2 Ulrich bundles on cubic threefolds. It is shown in \cite{Bea1} that such bundles, which have first Chern class equal to $2H,$ are parametrized by the complement of a divisor algebraically equivalent to $3\theta$ in the intermediate Jacobian.  For a variety $X \subseteq \mathbb{P}^{n}$ whose Picard number is 2 or greater, the problem of parametrizing Ulrich bundles of given rank on $X$ becomes more involved, partly because finding the divisor classes on $X$ which represent the first Chern class of an Ulrich bundle is less straightforward.      

Our main result characterizes these divisor classes when $X$ is a del Pezzo surface. This case has recently been studied in the papers \cite{CH, CH1, MP2} with a view to constructing indecomposable (e.g.~stable) Ulrich bundles of high rank.  It implies that Ulrich bundles of all ranks on del Pezzo surfaces, like their counterparts on curves, are fundamentally theta-divisorial in nature.
  
\begin{thm}
\label{main}
Let $X_{d} \subseteq \mathbb{P}^{d}$ be a del Pezzo surface of degree $d,$ and let $D \subseteq X_{d} \subseteq \mathbb{P}^{d}$ be a smooth connected curve of genus $g$ which does not lie in a hyperplane.  Then the following are equivalent:
\begin{itemize}
\item[(i)]{There exists an Ulrich bundle of rank $r$ on $X_{d}$ with $c_{1}(\mathcal{E})=D.$}
\item[(ii)]{The degree of $D$ is $dr,$ and for general smooth $C \in |D|,$ the kernel bundle $$M_{C}:=\ker(ev:H^{0}(\mathcal{O}_{X}(1)) \otimes \mathcal{O}_{C} \rightarrow \mathcal{O}_{C}(1))$$ admits a theta-divisor in $\textnormal{Pic}^{g-1+r}(C),$ i.e. there exists a line bundle $\mathcal{L}$ on $C$ of degree $g-1+r$ for which 
\begin{equation}
\label{mvanishing}
H^{0}(M_{C}\otimes \mathcal{L})=H^1(M_{C} \otimes \mathcal{L})=0.
\end{equation}}
\end{itemize}
\end{thm}

A brief explanation of (ii) is in order.  If $C \subseteq \mathbb{P}^{d}$ is a smooth irreducible curve of degree $dr,$ then we may define a {\it theta-divisor} associated to $M_C$ as follows: 
\begin{equation}
\Theta_{M_C}:=\{\mathcal{L} \in \textnormal{Pic}^{g-1+r}(C): H^0(M_C \otimes \mathcal{L}) \neq 0\}
\end{equation}
This set is nonempty, and it is a naturally a determinantal locus in $\textnormal{Pic}^{g-1+r}(C)$ whose expected codimension is 1.  Its codimension is 1 precisely when it is a proper subset of $\textnormal{Pic}^{g-1+r}(C),$ i.e. when there exists $\mathcal{L} \in \textnormal{Pic}^{g-1+r}(C)$ satisfying (\ref{mvanishing}).  

Since (ii) implies the semistability of $M_C,$ we have the following immediate consequence of Theorem \ref{main}:

\begin{cor}
\label{curvess}
If $\mathcal{E}$ is an Ulrich bundle on $X_{d}$ and $C$ is a general smooth member of $|\det{\mathcal{E}}|$ which is not contained in a hyperplane, then $M_{C}$ is semistable.
\end{cor}

Corollary \ref{curvess} can be used together with known Ulrich bundles on $X_d$ to produce curves on $X_{d}$ whose kernel bundle is semistable.  One source of such examples is hypersurface sections of $X_d \subseteq \mathbb{P}^d.$  For each $r \geq 2$ and $3 \leq d \leq 7,$ there exist Ulrich bundles of rank $r$ on $X_d$ with first Chern class $rH$ (see Proposition \ref{hypclass}, as well as \cite{MP2}). So Corollary \ref{curvess} implies that $M_{C}$ is semistable whenever $C$ is a general complete intersection of type $(3,r)$ in $\mathbb{P}^{3}$ or of type $(2,2,r)$ in $\mathbb{P}^{4}$ for all $r \geq 2.$  As far as we know, this is new for $r \geq 5;$ the case $r \leq 4$ is covered by the Theorem on p.~1-2 of \cite{PaRa}.     

Theorem \ref{main} also has implications for the study of the moduli space ${\rm SU}_{C}(d)$ of semistable vector bundles on $C$ with rank $d$ and trivial determinant.  One can show that there is a rational map ${\theta}:{\rm SU}_{C}(d) \dashrightarrow |d{\Theta}|$ which takes a general $E \in {\rm SU}_{C}(d)$ to the divisor 
\begin{equation}
{\Theta}_{E}:=\{\mathcal{L} \in {\rm Pic}^{g-1}(C) : H^0(E \otimes \mathcal{L}) \neq 0\}
\end{equation}
which is linearly equivalent to $d{\Theta}.$  (Here, $\Theta$ denotes Riemann's theta divisor in ${\rm Pic}^{g-1}(C).$)  We refer to \cite{Bea1, Po} for details. A consequence of the main theorem related to this is the following corollary.
\begin{cor}
If $C \subseteq X_d$ is a smooth irreducible curve of degree $dr,$ then the following are equivalent:
\begin{itemize}
\item[(i)]{There exists an Ulrich bundle $\mathcal{E}$ of rank $r$ on $X_d$ with $c_1(\mathcal{E})=C.$}
\item[(ii)]{For all $\eta \in {\rm Pic}^{r}(C)$ satisfying ${\eta}^{\otimes d} \cong \mathcal{O}_{C}(1),$ we have that $\theta$ is well-defined at (the S-equivalence class of) $M_{C} \otimes \eta.$}
\item[(iii)]{For some $\eta \in {\rm Pic}^{r}(C)$ satisfying ${\eta}^{\otimes d} \cong \mathcal{O}_{C}(1),$ we have that $\theta$ is well-defined at (the S-equivalence class of) $M_{C} \otimes \eta.$} 
\end{itemize}
\end{cor}  

In particular, if $C \subseteq X_d$ is a smooth irreducible curve of degree $dr$ and genus $g$ whose kernel bundle $M_C$ is semistable, and $C$ is \textit{not} the first Chern class of a rank-$r$ Ulrich bundle on $X_d,$ then the rational map $\theta$ has a base point at $M_{C} \otimes \eta$ for each $d$-th root $\eta$ of $\mathcal{O}_{C}(1).$  This will be explored further in future work. 

The semistability of $M_{C}$ was first studied in \cite{PaRa} with a view to Green's Conjecture on the syzygies of canonically embedded curves.  The connection stems from the fact that the syzygies of a smooth curve $C \subseteq \mathbb{P}^{n}$ admit a natural cohomological description in terms of the exterior powers of $M_C$ (see \cite{Laz} or (2.3.4) on p.~1-13 of \cite{PaRa}).  The stronger condition that $M_C$ admit a theta-divisor is related to a subtler aspect of the embedding of $C$.  

The Minimal Resolution Conjecture (MRC) for a subvariety $Y \subseteq \mathbb{P}^{N}$ implies that the minimal graded free resolution of a collection $\Gamma$ of sufficiently many general points on $Y$  is completely determined by that of $Y$ (see Section \ref{relminres} for a precise statement).  The case of a smooth curve $C \subseteq \mathbb{P}^{n}$ was studied by Farkas, Musta\c{t}\v{a} and Popa in \cite{FMP}, where they showed that MRC holds for canonically embedded curves.  An important part of their proof is the fact (stated here as Proposition \ref{fmpmrc}) that MRC holds for a curve $C \subseteq \mathbb{P}^{n}$ precisely when $M_{C}$ and its exterior powers admit theta-divisors in the appropriate Picard varieties.  We use this to obtain the following:

\begin{cor}
\label{minrescor}
Let $D$ be an effective divisor of degree $3r \geq 3$ on a smooth cubic surface $X_3 \subseteq \mathbb{P}^{3}.$  Then there exists an Ulrich bundle of rank $r$ on $X_3$ with first Chern class $D$ if and only if MRC holds for the general smooth curve $C$ in the linear system $|D|.$
\end{cor} 

It should be noted that a slightly different version of MRC for points on $X_d$ plays a major role in much of the recent work on Ulrich bundles on $X_d$ \cite{CH,MP1,MP2}.

One may ask whether Corollary \ref{minrescor} generalizes to $X_d$ for $4 \leq d \leq 9.$  To cleanly state the result we have obtained in this direction, we define the \textit{Ulrich semigroup} of a smooth projective variety $X$ in $\mathbb{P}^n$ to be
$$\mathfrak{Ulr}(X):=\{ D \in {\rm Pic}(X) : \textnormal{there exists an Ulrich bundle }\mathcal{E} \ \textnormal{on} \ X \ \textnormal{s.~t.} \ c_1(\mathcal{E})=D\}$$
The embedding of $X$ is suppressed in our notation since it will always be clear from context.  The extension of one Ulrich bundle by another is Ulrich (Proposition \ref{extcliff}), so $\mathfrak{Ulr}(X)$ is a sub-semigroup of ${\rm Pic}(X)$ if it is nonempty.

\begin{thm}
\label{mrchyp}
Suppose that for each generator $Q$ of $\mathfrak{Ulr}(X_d)$, MRC holds for a general smooth curve in the linear system $|Q|.$  

Then for an effective divisor  $D \subseteq X_d \subseteq \mathbb{P}^{d}$ which does not lie in a hyperplane, the following are equivalent:
\begin{itemize}
\item[(i)]{There exists an Ulrich bundle of rank $r$ on $X_d$ with $c_{1}(\mathcal{E})=D.$}
\item[(ii)]{The degree of $D$ is $dr,$ and MRC holds for a general smooth curve in the linear system $|D|.$}
\end{itemize} 
\end{thm}

Since MRC holds for any rational normal curve (Lemma \ref{ratnorm}), the hypothesis of Theorem \ref{mrchyp} is satisfied if $\mathfrak{Ulr}(X_d)$ is generated by classes of rational normal curves of degree $d.$  Theorem 3.9 of \cite{CH1} implies that this holds for $d=3,$ so Theorem \ref{mrchyp} yields a (rather convoluted) alternate proof of Corollary \ref{minrescor}.  

We round out this discussion with a brief overview of the case $d=9$, where the structure of the Ulrich semigroup is particularly simple.  Recall that the degree-9 del Pezzo surface $X_9$ is just the $3-$uple Veronese surface.  It admits a rank-2 Ulrich bundle with first Chern class $2H,$ and the symmetric square $\mathbf{S}^{2}T_{\mathbb{P}^2}$ of the tangent bundle of $\mathbb{P}^{2}$ is a stable rank-3 Ulrich bundle (Corollary 5.7 of \cite{ESW}).  It follows that $\mathfrak{Ulr}(X_9)$ is generated by $2H$ and $3H.$

Any smooth curve on $X_9$ which is a member of the linear system $|2H|$ is a nonhyperelliptic canonically embedded curve of genus 10, so MRC holds for such a curve, thanks to the main theorem of \cite{FMP}.  Consequently, all that is needed for $X_9$ to satisfy the hypotheses of Theorem \ref{mrchyp} is for MRC to hold for the general smooth member of $|3H|.$  Curves of the latter type are half-canonical curves of genus 28, and it is still not known whether MRC holds for them.

Given that the theory of maximal Cohen-Macaulay modules over Gorenstein rings is particularly robust \cite{Buch}, it is natural to ask whether Theorem \ref{main} generalizes to arithmetically Gorenstein surfaces that are not del Pezzo.  Proposition \ref{onlydelpezzo} answers this negatively in the following sense: if $C$ is a smooth irreducible curve of genus $g$ on a smooth arithmetically Gorenstein surface $X,$ then roughly speaking, a line bundle of degree $g-1+r$ on $C$ can give rise to an Ulrich bundle of rank $r$ on $X$ only if $X$ is a del Pezzo or a quadric surface.  Moreover, every curve on a quadric surface $Q$ represents the first Chern class of an Ulrich bundle on $Q$ (Proposition \ref{everycurve}), so there is no analogue of Theorem \ref{main} in the quadric surface case either. 

Finally, while the present article has little to say about stable Ulrich bundles, we will investigate in future work the extent to which stability of Ulrich bundles is reflected in the geometry of the associated theta-divisors.

\bigskip

\textbf{Acknowledgments:} The second author was partially supported by the NSF grants DMS-0603684 and DMS-1004306, and the third author was partially supported by the NSF grant RTG DMS-0502170.  We are grateful to D. Eisenbud, D. Faenzi, G. Farkas, R. Hartshorne, R. Lazarsfeld, and M. Popa for valuable discussions and correspondence related to this work, and to R. Mir\'{o}-Roig and J. Pons-Llopis for sharing their preprints \cite{MP1} and \cite{MP2} with us.   

\bigskip

\textbf{Conventions:}  We work over an algebraically closed field of characteristic zero.  All del Pezzo surfaces are of degree $d \geq 3$ and are embedded in $\mathbb{P}^{d}$ by their anticanonical series.

\bigskip

\section{Generalities on Ulrich Bundles}
\label{gencliffbun}

\subsection{First Properties}  

This subsection contains a summary of general properties of Ulrich bundles which will be used in the sequel.  While some of the proofs appear in \cite{CKM1}, we reproduce them here for the reader's convenience.  

Throughout this subsection, $X \subseteq \mathbb{P}^{n}={\rm Proj}(S)$ is a smooth projective variety of degree $d \geq 2$ and dimension $k \geq 1,$ and $\iota: X \hookrightarrow \mathbb{P}^{n}$ denotes inclusion.  

We begin by restating the definition given at the beginning of the Introduction.

\begin{defn}
A vector bundle $\mathcal{E}$ on $X$ is an \textit{Ulrich bundle} if the graded module $\Gamma_{\ast}(\mathcal{E})$ is an Ulrich $S-$module.
\end{defn}

The following characterization of Ulrich bundles appears as Proposition 2.1 in \cite{ESW}; we state it without proof.
 
\begin{prop}
\label{cliffdef}
Let $\mathcal{E}$ be a vector bundle of rank $r$ on $X.$  Then the following are equivalent:
\begin{itemize}
\item[(i)]{$\mathcal{E}$ is Ulrich.}
\item[(ii)]{$\iota_{\ast}\mathcal{E}$ admits a minimal free resolution of the form
\begin{equation}
\label{linres}
0 \leftarrow \iota_{\ast}\mathcal{E} \leftarrow \mathcal{O}_{\mathbb{P}^{n}}^{\beta_{0}} \xleftarrow{\phi_{0}} \mathcal{O}_{\mathbb{P}^{n}}(-1)^{\beta_{1}} \xleftarrow{\phi_{1}} \cdots \xleftarrow{\phi_{n-k-1}} \mathcal{O}_{\mathbb{P}^{n}}(-(n-k))^{\beta_{n-k}} \leftarrow 0
\end{equation} where ${\beta}_i = dr\binom{n-k}{i}.$}
\item[(iii)]{For all linear projections $\pi:X \rightarrow \mathbb{P}^{k},$ we have $\pi_{\ast}\mathcal{E} \cong \mathcal{O}_{\mathbb{P}^{k}}^{dr}.$}
\item[(iv)]{For some linear projection $\pi:X \rightarrow \mathbb{P}^{k}$ we have $\pi_{\ast}\mathcal{E} \cong \mathcal{O}_{\mathbb{P}^{k}}^{dr}.$ \hfill \qedsymbol}
\item[(v)]{$H^i(\mathcal{E}(-i))=0$ for $i > 0$ and $H^j(\mathcal{E}(-j-1))=0$ for $j < k.$}
\end{itemize}
\end{prop}

The next result is particularly important to the proof of Theorem \ref{main}.  

\begin{prop}
\label{ulrichchar}
Let $\mathcal{E}$ be a vector bundle of rank $r$ on $X,$ where $X$ is of dimension $k \geq 2.$  Then $\mathcal{E}$ is Ulrich if and only if it is ACM with Hilbert polynomial $dr\binom{t+k}{k}.$
\end{prop}

\begin{proof}
($\Rightarrow$) This is Corollary 2.2 in \cite{ESW}. 

($\Leftarrow$) Let $\mathcal{E}$ be an ACM vector bundle of rank $r$ on $X$ with Hilbert polynomial $dr\binom{t+k}{k},$ and let $\pi:X \rightarrow \mathbb{P}^{k}$ be a linear projection.  Then $\pi_{\ast}\mathcal{E}$ is an ACM vector bundle of rank $dr$ on $\mathbb{P}^{k},$ and Horrocks' Theorem implies that $\pi_{\ast}\mathcal{E} \cong \bigoplus_{j=1}^{dr}\mathcal{O}_{\mathbb{P}^{k}}(n_i)$ for some $n_{1}, \cdots n_{dr} \in \mathbb{Z}.$  It follows from the projection formula that $h^0(\pi_{\ast}\mathcal{E}(t))=h^0(\mathcal{E}(t))$ for all $t \in \mathbb{Z},$ so we have the following equality of Hilbert polynomials:
\begin{equation}
\label{hilbequality}
dr\binom{t+k}{k}=\sum_{i=1}^{dr}\binom{t+n_{i}+k}{k}
\end{equation}
Equating the coefficents of $t^{k-1}$ on either side, we have that $\sum_{i=1}^{dr}n_{i}=0.$  We may use this to deduce, after equating the coefficients of $t^{k-2}$ on either side, that 
\begin{equation}
\sum_{1 \leq j < \ell \leq k}dr \cdot j{\ell} =\sum_{i=1}^{dr}\Biggl(\sum_{1 \leq j < \ell \leq k}(n_{i}+j)(n_{i}+\ell)\Biggr) = \binom{k}{2}\sum_{i=1}^{dr}n_{i}^{2} + \sum_{1 \leq j < \ell \leq k}dr \cdot j{\ell}.
\end{equation} 
We then have that $\sum_{i=1}^{dr}n_{i}^{2}=0,$ which implies that $n_{i}=0$ for all $i.$  By (iii) of Proposition \ref{cliffdef}, we have that $\mathcal{E}$ is Ulrich.
\end{proof}

\begin{rem}
\label{curverem}
This is false when $k=1.$  Indeed, the ACM condition is vacuous for vector bundles on curves, and as mentioned in the Introduction, there exist for each curve $X \subseteq \mathbb{P}^{n}$ of degree $d$ and genus $g$ line bundles with Hilbert polynomial $d(t+1)$ (i.e. degree $g-1+d$) which are not Ulrich, namely those which lie in the theta-divisor associated to $\mathcal{O}_{X}(-1).$ 
\end{rem}

The following is a straightforward consequence of Proposition \ref{cliffdef}.

\begin{cor}
\label{cliffbasics}
Let $\mathcal{E}$ be an Ulrich bundle of rank $r$ on $X.$  Then we have the following:
\begin{itemize}
\item[(i)]{$\mathcal{E}$ is globally generated.}
\item[(ii)]{$h^{0}(\mathcal{E})=\chi(\mathcal{E})=dr.$}
\item[(iii)]{$\mathcal{E}$ is normalized, i.e. $H^{0}(\mathcal{E}) \neq 0$ and $H^{0}(\mathcal{E}(-1))=0.$ \hfill \qedsymbol}
\end{itemize}
\end{cor}

\begin{prop}
\label{giesstab}
Let $\mathcal{E}$ be an Ulrich bundle of rank $r \geq 1$ on $X$. Then $\mathcal{E}$ is semistable.
\end{prop}

\begin{proof}
Let $\mathcal{F}$ be a rank-$s$ torsion-free coherent subsheaf of $\mathcal{E}$.  Then $\pi_{\ast}\mathcal{F}$ is a rank-$ds$ torsion-free coherent subsheaf of $\pi_{\ast}\mathcal{E}=\mathcal{O}_{\mathbb{P}^{n-1}}^{dr},$ and since $\mathcal{O}_{\mathbb{P}^{n-1}}^{dr}$ is semistable, we have that $p(\pi_{\ast}\mathcal{F}) \leq p(\pi_{\ast}\mathcal{E})$.  Since cohomology is preserved under finite pushforward, we have that $d \cdot p(\pi_{\ast}\mathcal{F})=p(\mathcal{F})$ and $d \cdot p(\pi_{\ast}\mathcal{E})=p(\mathcal{E})$.  It follows immediately that $p(\mathcal{F}) \leq p(\mathcal{E})$.
\end{proof}

It will be important to know that the Ulrich property is well behaved in short exact sequences (Proposition \ref{extcliff}). First, we need a lemma.

\begin{lem}
\label{locfree}
Let $g:Y \rightarrow Z$ be a finite flat surjective morphism of smooth projective varieties, and let $\mathcal{G}$ be a coherent sheaf on $Y$ such that $g_{\ast}\mathcal{G}$ is locally free.  Then $\mathcal{G}$ is locally free.
\end{lem}
\begin{proof}
To show that $\mathcal{G}$ is locally free, we show that the stalks this sheaf are free modules over the local ring at any point. So translating the hypotheses into the local situation, we have a finite flat morphism of regular local rings $R \rightarrow S$ and a  finite $S$-module $M$ such that $M$ as an $R$-module is locally free of finite rank.  Thus $\displaystyle {\rm Ext}^i_R(M, R) = 0 \ \textnormal{for any} \  i > 0$ and $\displaystyle {\rm Hom}_R(M, R) \cong S$ as an $S$-module. We also have the change of rings spectral sequence:
\[
\displaystyle {\rm Ext}^i_S(M, {\rm Ext}^j_R(S, R)) \Rightarrow {\rm Ext}^{i + j}_R(M, R).
\]
\noindent The degeneration of this spectral sequence gives the isomorphism
\[
\displaystyle {\rm Ext}^i_S(M, S) \cong  {\rm Ext}^{i }_R(M, R) = 0
\]
\noindent for any $i >0$. So $M$ is a free $R$ module. 
\end{proof}

\begin{prop}
\label{extcliff}
Consider the following short exact sequence of coherent sheaves on $X$:
\begin{equation}
0 \rightarrow \mathcal{F} \rightarrow \mathcal{E} \rightarrow \mathcal{G} \rightarrow 0
\end{equation}
If any two of $\mathcal{F}$, $\mathcal{E}$, and $\mathcal{G}$ are Ulrich bundles, then so is the third.
\end{prop}

\begin{proof}
Let $f, e$, and $e-f$ be the respective ranks of $\mathcal{F},$ $\mathcal{E},$ and $\mathcal{G}.$  Since $\pi$ is a finite morphism, we have the following exact sequence of sheaves on $\mathbb{P}^{n-1}$:
\begin{equation}
\label{cliffpush}
0 \rightarrow \pi_{\ast}\mathcal{F} \rightarrow \pi_{\ast}\mathcal{E} \rightarrow \pi_{\ast}\mathcal{G} \rightarrow 0.
\end{equation}
If $\mathcal{F}$ and $\mathcal{G}$ are Ulrich bundles, then $\pi_{\ast}\mathcal{F}$ and $\pi_{\ast}\mathcal{G}$ are trivial vector bundles on $\mathbb{P}^{n-1}.$  Therefore $\pi_{\ast}\mathcal{E}$, being an extension of trivial vector bundles on $\mathbb{P}^{n-1}$, is also trivial, so that $\mathcal{E}$ is Ulrich.

If $\mathcal{E}$ and $\mathcal{G}$ are Ulrich bundles, then $\mathcal{F}$ is locally free.  By definition $\pi_{\ast}\mathcal{E}$ and $\pi_{\ast}\mathcal{G}$ are trivial, so dualizing (\ref{cliffpush}) yields the exact sequence
\begin{equation}
0 \rightarrow \mathcal{O}_{\mathbb{P}^{n-1}}^{d(e-f)} \rightarrow \mathcal{O}_{\mathbb{P}^{n-1}}^{de} \rightarrow (\pi_{\ast}\mathcal{F})^{\vee}\rightarrow 0
\end{equation}
It follows from taking cohomology that $(\pi_{\ast}\mathcal{F})^{\vee}$ is a globally generated vector bundle of rank $df$ on $\mathbb{P}^{n-1}$ with exactly $df$ global sections, so it must be trivial.  In particular, $\pi_{\ast}\mathcal{F} \cong \mathcal{O}_{\mathbb{P}^{n-1}}^{dr}$, i.e. $\mathcal{F}$ is Ulrich.

Finally, if $\mathcal{F}$ and $\mathcal{E}$ are Ulrich bundles, then applying the functor $\mathcal{H}om({\cdot},\mathcal{O}_{\mathbb{P}^{n-1}})$ to (\ref{cliffpush}) shows that the globally generated coherent sheaf $\pi_{\ast}\mathcal{G}$ is torsion-free of rank $d(e-f).$  Since $\pi_{\ast}\mathcal{G}$ has exactly $d(e-f)$ global sections, it follows that $\pi_{\ast}\mathcal{G}$ is trivial.  Lemma \ref{locfree} then implies that $\mathcal{G}$ is locally free, hence an Ulrich bundle.
\end{proof}


\subsection{The case of Arithmetically Gorenstein surfaces}
\label{gorenstein}
In order to prepare for the proofs of our main results and place them in a proper context, we embark on a study of Ulrich bundles on a smooth AG (Arithmetically Gorenstein) surface.  Recall that a smooth projective variety $X \subseteq \mathbb{P}^{n}$ is AG if it is ACM and its canonical bundle is isomorphic to $\mathcal{O}_{X}(m)$ for some $m \in \mathbb{Z}.$

\begin{lem}
\label{mbound}
If $X \subseteq \mathbb{P}^{n}$ is an AG surface with $K_{X}=mH,$ then $m \geq -2,$ with equality if and only if $X$ is a smooth quadric surface.
\end{lem}

\begin{proof}
Suppose that $K_{X}=mH$ for $m \leq -3,$ and let $g$ be the (nonnegative) arithmetic genus of a hyperplane section of $X.$  The adjunction formula implies that $2g-2 \leq -2H^{2}.$  Since $g$ is nonnegative, we have $H^{2} \leq 1,$ which is impossible if $H$ is very ample.

It is clear that $m=-2$ if $X$ is a smooth quadric in $\mathbb{P}^{3}.$  Conversely, if $K_X=-2H,$ then the adjunction formula implies that $2g-2=-H^{2},$ which is only possible if $H^{2}=2.$  Since $-K_{X}$ is ample, Riemann-Roch tells us that the linear system $|H|$ is 3-dimensional, so that $X$ is a quadric surface.
\end{proof}

\begin{prop}
\label{agchern}
Suppose $X \subseteq \mathbb{P}^{n}$ is an AG surface of degree $d \geq 2$ with hyperplane class $H$ and $K_X=mH.$  Then a vector bundle $\mathcal{E}$ of rank $r$ on $X$ is Ulrich if and only if the following conditions hold:
\begin{itemize}
\item[(i)]{$\mathcal{E}$ is ACM.}
\item[(ii)]{$c_1(\mathcal{E}) \cdot H = \bigl(\frac{m+3}{2}\bigr) \cdot dr.$}
\item[(iii)]{$c_2(\mathcal{E})=\frac{c_1(\mathcal{E})^{2}}{2}-dr \cdot \bigl(\frac{m^{2}+3m+4}{4}\bigr) +r \cdot (1+h^0(K_X)).$}
\end{itemize}
\end{prop}

\begin{proof}
($\Rightarrow$) If $\mathcal{E}$ is Ulrich, it satisfies (i), so we need only compute the Chern classes of $\mathcal{E}.$  Assume $t \in \mathbb{Z}$ is sufficiently large so that $h^0(\mathcal{E}(t))$ coincides with the Hilbert polynomial of $\mathcal{E}$ evaluated at $t.$  Since $X$ is AG, we have $h^1(\mathcal{O}_{X})=0,$ so Riemann-Roch implies that
\begin{equation}
\label{agrr}
h^0(\mathcal{E}(t))=\frac{c_1(\mathcal{E}(t))^{2}}{2}-c_2(\mathcal{E}(t))-\frac{c_1(\mathcal{E}(t)) \cdot K_X}{2}+r(1+h^0(K_X))
\end{equation}
$$=\biggl(\frac{dr}{2}\biggr)t^2 + \biggl(c_1(\mathcal{E}) \cdot H-\frac{mdr}{2}\biggr)t + \chi(\mathcal{E})$$
It follows from (ii) of Corollary \ref{cliffbasics} and (vi) of Proposition \ref{cliffdef}, respectively, that  $\chi(\mathcal{E})=h^0(\mathcal{E})=dr$ and that the Hilbert polynomial of $\mathcal{E}$ is $dr\binom{t+2}{2}=\bigl(\frac{dr}{2}\bigr)t^2+\bigl(\frac{3dr}{2}\bigr)t+dr.$  Expanding out $\chi(\mathcal{E})$ by Riemann-Roch and equating coefficients, we obtain the desired values of $c_1(\mathcal{E}) \cdot H$ and $c_2(\mathcal{E}).$

($\Leftarrow$)  Assume that $\mathcal{E}$ satisfies conditions (i),(ii), and (iii).  Combining Riemann-Roch with (ii) and (iii), we see that $\chi(\mathcal{E})=dr.$  The computation (\ref{agrr}) is still valid under our hypothesis, so we may set $\chi(\mathcal{E})=dr$ and $c_1(\mathcal{E}) \cdot H = \bigl(\frac{m+3}{2}\bigr) \cdot dr$ in (\ref{agrr}) and deduce that the Hilbert polynomial of $\mathcal{E}$ is $dr\binom{t+2}{2},$ i.e. that $\mathcal{E}$ is an Ulrich bundle.
\end{proof} 

\begin{prop}
\label{agdual}
Let $\mathcal{E}$ be an Ulrich bundle of rank $r$ on an AG surface $X.$  If $K_{X}=mH,$ then $\mathcal{E}^{\vee}(m+3)$ is also Ulrich.
\end{prop}

\begin{proof}
Given that $h^1(\mathcal{E}^{\vee}(t+m+3))=h^1(\mathcal{E}(-t-3))=0$ for all $t \in \mathbb{Z}$ by Serre duality and the hypothesis that $\mathcal{E}$ is Ulrich, we have that $\mathcal{E}^{\vee}(m+3)$ satisfies condition (i) of Proposition \ref{agchern}.  The latter result implies that $c_1(\mathcal{E}) \cdot H = \bigl(\frac{m+3}{2}\bigr)dr;$ consequently, we have 
\begin{equation}
c_1(\mathcal{E}^{\vee}(m+3)) \cdot H = dr(m+3)-c_1(\mathcal{E}) \cdot H = \biggl(\frac{m+3}{2}\biggr)dr
\end{equation}
\begin{equation}
\frac{c_1(\mathcal{E}^{\vee}(m+3))^{2}}{2}=\frac{c_1(\mathcal{E})^{2}}{2}-(m+3)rc_1(\mathcal{E}) \cdot H+\biggl(\frac{(m+3)^2}{2}\biggr)dr = \frac{c_1(\mathcal{E})^{2}}{2}.
\end{equation}
It follows at once that $\mathcal{E}^{\vee}(m+3)$ satisfies conditions (ii) and (iii) in Proposition \ref{agchern}.  This concludes the proof.
\end{proof}

\begin{prop}
\label{repcurve}
If $\mathcal{E}$ is an Ulrich bundle of rank $r$ on an AG surface $X,$ then $c_1(\mathcal{E})$ is represented by a smooth curve on $X.$
\end{prop}

\begin{proof}
Since $\mathcal{E}$ is Ulrich, it is globally generated by (i) of Corollary \ref{cliffbasics}, and thus $\det{\mathcal{E}}$ is globally generated as well.  By Lemma \ref{mbound} and (ii) of Proposition \ref{agchern}, we have that $c_1(\mathcal{E}) \cdot H > 0,$ so $\det{\mathcal{E}}$ cannot be the trivial line bundle.  This implies that the linear system $|\det{\mathcal{E}}|$ is basepoint-free of dimension 1 or greater; the result then follows from Bertini's theorem.  
\end{proof}

To motivate the next result, we give a brief account of the central construction in the proof of Theorem \ref{main}, which has appeared in the literature under the names \textit{elementary modification} (e.g. \cite{Fri}) and \textit{elementary transformation} (e.g. \cite{HL}).  If $C$ is a smooth irreducible curve of genus $g$ contained in a surface $X,$ and $\mathcal{L}$ is a globally generated line bundle on $C$ with $r \geq 1$ global sections, then the rank-$r$ vector bundle $\mathcal{E}$ which occurs in the exact sequence
\begin{equation}
\label{elm1}
0 \rightarrow \mathcal{E}^{\vee} \rightarrow H^0(\mathcal{L}) \otimes \mathcal{O}_{X} \rightarrow \mathcal{L} \rightarrow 0
\end{equation}
satisfies $c_1(\mathcal{E})=[C]$ and $c_2(\mathcal{E})=c_1(\mathcal{L}).$  One expects from Riemann-Roch that $c_1(\mathcal{L})=g-1+r,$ and this is reasonable since the general line bundle of degree $g-1+r$ is globally generated (Lemma \ref{bpfln}).  Given that we are interested in the case where $\mathcal{E}$ is Ulrich, it makes sense to determine when the second Chern class computed in Proposition \ref{agchern} is equal to $g-1+r.$

Observe that the hypothesis below on the first Chern class is satisfied whenever $X$ admits an Ulrich bundle, thanks to Proposition \ref{repcurve}, and that the smooth curve representing this first Chern class is not assumed to be connected.
\begin{prop}
\label{onlydelpezzo}
Let $X \subseteq \mathbb{P}^{n}$ be an AG surface of degree $d \geq 2,$ and suppose that $X$ admits a Ulrich bundle $\mathcal{E}$ of rank $r$ on $X$ whose first Chern class $c_1(\mathcal{E})$ is represented by a smooth curve $C$ of arithmetic genus $g.$ 

If $c_2(\mathcal{E})= g-1+r,$ then $X$ is either a smooth quadric surface in $\mathbb{P}^{3}$ or a del Pezzo surface.
\end{prop}

\begin{proof}
Let $m$ be the integer for which $K_{X}=mH.$  Since $2g-2=c_1(\mathcal{E})^{2}+c_1(\mathcal{E}) \cdot K_X=c_1(\mathcal{E})^{2}+\bigl(\frac{m^{2}+3m}{2}\bigr)dr,$ we have from (iii) of Proposition \ref{agchern} that $c_2(\mathcal{E})=g-1+r$ if and only if
\begin{equation}
\label{canonicalsections}
h^0(K_{X})=\frac{d(m+1)(m+2)}{2} 
\end{equation}  
By Lemma \ref{mbound}, we need only consider $m \geq -2.$  The condition (\ref{canonicalsections}) clearly holds for $m=-2$ and $m=-1.$  In the latter case, $X$ is a del Pezzo surface, and in the former case, Lemma \ref{mbound} implies that $X$ is a smooth quadric surface in $\mathbb{P}^{3}.$  It remains to show that (\ref{canonicalsections}) cannot hold for $m \geq 0.$  

Assume (\ref{canonicalsections}) holds for some $m \geq 0.$  If $m=0,$ then $K_X$ is the trivial bundle, and (\ref{canonicalsections}) reduces to $d=1,$ which is impossible.  If $m \geq 1,$ then $K_X$ is ample; this implies $X$ is a minimal surface of general type.  We then have from Noether's inequality that 
\begin{equation}
\frac{d(m+1)(m+2)}{2} = h^0(K_X) \leq \frac{K_X^{2}}{2}+2 = \frac{dm^{2}}{2}+2.
\end{equation}
After subtracting $\frac{dm^2}{2}$ from both sides, this reduces to $d\bigl(\frac{3m}{2}+1\bigr) \leq 2.$  Since $d \geq 2$ and $m \geq 1,$ we have a contradiction. 
\end{proof}

The following result, together with Proposition \ref{agchern}, will help us classify Ulrich line bundles on both quadric surfaces and del Pezzo surfaces.

\begin{prop}
\label{acmcurve}
Let $X \subseteq \mathbb{P}^{n}$ be an AG surface.  If $C \subseteq X$ is an ACM curve, then $\mathcal{O}_{X}(-C)$ and $\mathcal{O}_{X}(C)$ are ACM line bundles on $X.$
\end{prop}

\begin{proof}
By Serre duality and our hypothesis on $X$ in $\mathbb{P}^{n},$ it suffices to check that $\mathcal{O}_{X}(-C)$ is ACM.  Consider the exact sequence
\begin{equation}
0 \rightarrow \mathcal{I}_{X|\mathbb{P}^{n}} \rightarrow \mathcal{I}_{C|\mathbb{P}^{n}} \rightarrow \mathcal{O}_{X}(-C) \rightarrow 0
\end{equation}
We are now reduced to showing that $H^{1}(\mathcal{I}_{C|\mathbb{P}^{n}}(t))=0$ and $H^{2}(\mathcal{I}_{X|\mathbb{P}^{n}}(t))=0$ for all $t \in \mathbb{Z}.$  The first set of vanishings follows immediately from the fact that $C$ is ACM.  To verify the second set of vanishings, note that for all $t \in \mathbb{Z}$ the cohomology group $H^2(\mathcal{I}_{X|\mathbb{P}^{n}}(t))$ is a quotient of $H^{1}(\mathcal{O}_{X}(t)),$ and the latter is always zero since $X$ is ACM.
\end{proof}

We now characterize curves on a quadric surface which represent first Chern classes of line bundles. 

\begin{prop}
\label{everycurve}
If $X \subseteq \mathbb{P}^{3}$ is a smooth quadric surface, then every curve on $X$ represents the first Chern class of an Ulrich bundle on $X.$
\end{prop}
 
\begin{proof}
Let $F_1$ and $F_2$ be lines on $X$ whose classes generate ${\rm Pic}(X).$  Since $F_1$ and $F_2$ are complete intersections in $\mathbb{P}^{3},$ they are ACM curves, so Proposition \ref{acmcurve} implies that $\mathcal{O}_{X}(F_1)$ and $\mathcal{O}_{X}(F_2)$ are both ACM line bundles on $X.$  Moreover, for $i=1,2,$ the degree and second Chern class of $\mathcal{O}_{X}(F_{i})$ coincide with (ii) and (iii) in Proposition \ref{agchern}, respectively (after setting $r=1, d=2,m=-2,$ and $h^0(K_{X})=0$), so the latter result implies that $\mathcal{O}_{X}(F_{1})$ and $\mathcal{O}_{X}(F_{2})$ are Ulrich line bundles on $X.$

If $C$ is a curve on $X,$ then $C$ is linearly equivalent to $m_{1}F_{1}+m_{2}F_{2}$ for nonnegative integers $m_{1},m_{2}$ which are not both zero, so it represents the first Chern class of the vector bundle $\mathcal{O}_{X}(F_{1})^{\oplus m_{1}} \oplus \mathcal{O}_{X}(F_{2})^{\oplus m_{2}},$ which is Ulrich by Proposition \ref{extcliff}.
\end{proof}

\begin{rem}
The line bundles $\mathcal{O}_{X}(F_{1}-H)$ and $\mathcal{O}_{X}(F_{2}-H)$ are the spinor bundles on $X.$  More generally, if $X \subset \mathbb{P}^{n}$ is a smooth quadric hypersurface and $\mathcal{E}$ is a spinor bundle on $X,$ then $\mathcal{E}(1)$ is a stable Ulrich bundle on $X$; see \cite{Ott} for details. 
\end{rem}

\subsection{The case of del Pezzo surfaces}
In this subsection, we specialize the results of the previous section to the del Pezzo case.  From this section on, $X_{d}$ denotes a del Pezzo surface of degree $d \leq 9,$ and $H$ denotes the hyperplane class on $X.$  

The following result is a straightforward consequence of Proposition \ref{agchern}.

\begin{prop}
\label{chern}
A vector bundle $\mathcal{E}$ of rank $r$ on $X_{d}$ is Ulrich if and only if the following conditions are satisfied:
\begin{itemize}
\item[(i)]{$\mathcal{E}$ is ACM, and $h^0(\mathcal{E}(-1))=0.$}
\item[(ii)]{$c_{1}(\mathcal{E}) \cdot H=dr.$}
\item[(iii)]{$c_{2}(\mathcal{E})=\frac{c_1(\mathcal{E})^2-(d-2)r}{2}.$}
\end{itemize} \hfill \qedsymbol
\label{prop:Ulrichcharacterization}
\end{prop}

Note that statement (ii) below is false in the quadric surface case.

\begin{lem}
\label{numprop}
Let $\mathcal{E}$ be an Ulrich bundle of rank $r$ on $X_{d}.$  Then we have the following:
\begin{itemize}
\item[(i)]{$\det{\mathcal{E}}$ is globally generated.  In particular, $c_1(\mathcal{E})$ is nef.}
\item[(ii)]{The general member of the linear system $|\det{\mathcal{E}}|$ is smooth and irreducible.}
\end{itemize}
\end{lem}

\begin{proof}
(i) follows immediately from (i) of Corollary \ref{cliffbasics}.  If the degree $c_1(\mathcal{E})^2$ of the morphism $f:X_d \rightarrow |\det{\mathcal{E}}|^{\ast}$ associated to $\det{\mathcal{E}}$ is positive, then (ii) follows from applying the Bertini and Lefschetz theorems to the finite part of the Stein factorization of $f.$  The fact that $\mathcal{E}$ is globally generated implies that $c_2(\mathcal{E}) \geq 0,$ which implies in turn (via Corollary \ref{chern}) that $c_1(\mathcal{E})^{2} \geq (d-2)r > 0;$ thus (ii) is proved.
\end{proof}

\begin{prop}
Let $\mathcal{L}$ be a line bundle on a degree-$d$ del Pezzo surface $X_{d}.$  Then $\mathcal{L}$ is an Ulrich line bundle on $X_{d}$ if and only if $\mathcal{L} \cong \mathcal{O}_{X_{d}}(Q),$ where $Q$ is the class of a rational normal curve on $X_{d}.$
\end{prop}
\begin{proof}
($\Rightarrow$) Let $\mathcal{L}$ be an Ulrich line bundle on $X_d.$   Since ${\chi}(\mathcal{L})=h^{0}(\mathcal{L})=d,$ we have from Riemann-Roch that $c_{1}(\mathcal{L})^{2}=d-2.$  If $Q$ is a smooth irreducible member of $|\mathcal{L}|$ (which exists by Lemma \ref{numprop}), then the adjunction formula implies that the genus of $Q$ is equal to 0.  Since the degree of $Q$ is $d,$ it follows that $Q$ is a rational normal curve.

($\Leftarrow$) Fix a rational normal curve $Q \subseteq X_{d}.$  Since $Q$ is of degree $d$ and $c_{2}(\mathcal{O}_{X_{d}}(Q))=0,$ it suffices by Corollary \ref{chern} to check that $\mathcal{O}_{X_{d}}(Q)$ is ACM.  But this follows from Proposition \ref{acmcurve}.
\end{proof}

$X_d$ does not contain any rational normal curves of degree $d$ whenever $d=8$ or $d=9,$ so the following is immediate.

\begin{cor}
$X_d$ does not admit an Ulrich line bundle when $d=8$ or $d=9.$ \hfill \qedsymbol
\end{cor} 

\textit{Remark:}  There exist rank-2 Ulrich bundles on cubic surfaces whose first Chern class cannot be represented by a smooth ACM curve.  Indeed, any smooth curve $C$ representing the first Chern class of a bundle of type (A.3) in Theorem 1 of \cite{Fae} is of genus 2 and degree 6 in $\mathbb{P}^3.$  Riemann-Roch implies that $C$ is not linearly normal, and therefore not ACM.      

\begin{prop}
\label{csquared}
Let $X \subseteq \mathbb{P}^{d}$ be a del Pezzo surface of degree $d \geq 3,$ and let $\mathcal{E}$ be an Ulrich bundle on $X$ of rank $r \geq 1.$  Then
\begin{equation}
\label{ineq}
(d-2)r^{2} \leq c_{1}(\mathcal{E})^{2} \leq dr^{2}
\end{equation}
\end{prop}

\begin{proof}
Let $\mathcal{E}$ be an Ulrich bundle on $X_d$ of rank $r.$  Then the Hodge Index Theorem implies that $c_1(\mathcal{E})^{2} \leq \frac{(c_1(\mathcal{E}) \cdot H)^2}{H^2}=dr^{2}.$ 

By Proposition \ref{giesstab}, $\mathcal{E}$ is semistable in the sense of Gieseker, so it is also semistable in the sense of Mumford.  We may then use Bogomolov's inequality (e.g. Theorem 3.4.1 in \cite{HL}) to deduce that $\frac{r-1}{2r}{c_1(\mathcal{E})}^{2} \leq c_2(\mathcal{E}).$  Combining this with the identity $c_1(\mathcal{E})^{2}=2c_2(\mathcal{E})+(d-2)r$ gives the desired inequality $c_1(\mathcal{E})^{2} \geq (d-2)r^{2}.$  
\end{proof}

The lower bound in (\ref{ineq}) is sharp for $3 \leq d \leq 7.$  Indeed, if $Q$ is any rational normal curve on $X_d$ for $d$ in this range, then the rank-$r$ vector bundle $\mathcal{E}:=\mathcal{O}_{X_d}(Q)^{\oplus r}$ is Ulrich with first Chern class $rQ,$ so that $c_1(\mathcal{E})^{2}=(d-2)r^2.$  The sharpness of the upper bound will be discussed in Proposition \ref{hypclass}.  

\section{Generalized Theta-Divisors and Ulrich Bundles}

This section contains the proof of Theorem \ref{main}, which is an immediate consequence of Theorems \ref{mainprop1} and \ref{mainprop2}.  We begin with a brief review of the generalized theta-divisors defined in the Introduction.  

\subsection{Generalized Theta-Divisors}
One of the pillars of the theory of algebraic curves is the theta-divisor associated to a smooth irreducible projective curve of genus $g,$ i.e. the locus of line bundles of degree $g-1$ on $C$ which admit global sections (cf. Chapter I of \cite{ACGH}).  The following  generalization has proven to be very fruitful (cf. \cite{Bea2}, \cite{Po} and the references therein).
\begin{defn}
Let $F$ be a vector bundle on $C$ with rank $r$ and degree $r(g-1-d).$  We say that $F$ \textit{admits a theta-divisor in }$\textnormal{Pic}^{d}(C)$ if for some $L \in \textnormal{Pic}^{d}(C)$ we have that $h^{0}(F \otimes L)=0.$
\end{defn}
Since the vanishing of $h^0(F \otimes L)$ is an open condition on $\textnormal{Pic}^{d}(C),$ this condition amounts to saying that the locus
\begin{equation}
\Theta_{F}:=\{L \in \textnormal{Pic}^{d}(C): h^{0}(F \otimes L) \neq 0\}
\end{equation}
is a proper Zariski-closed subset of $\textnormal{Pic}^{d}(C).$  Moreover, the expected codimension of $\Theta_{F}$ is equal to 1, so $F$ admits a theta-divisor in $\textnormal{Pic}^{d}(C)$ precisely when $\Theta_{F}$ is a divisor.  The reason that $\Theta_{F}$ is nonempty is that it is the inverse image of the (ample and effective) theta-divisor on $U_{C}(r,r(g-1))$ under the natural morphism $t_{F}: \textnormal{Pic}^{d}(C) \rightarrow U_{C}(r,r(g-1))$ defined by $L \mapsto F \otimes L.$

While every line bundle of degree $g-1-d$ admits a theta-divisor in $\textnormal{Pic}^{d}(C),$ there are examples of higher-rank vector bundles of slope $g-1-d$ which do not; see Section 6.2 of \cite{Po} for details.

\subsection{Proof of Theorem \ref{main}}
We begin by proving the implication (i)$\Rightarrow$(ii).  The first result in this section implies that to verify (ii), we need only produce a single smooth curve $C$ in the relevant linear system for which $M_C$ admits a theta-divisor.  It will also be used in the proof of Theorem \ref{mrchyp}. 

\begin{prop}
\label{thetaopen}
Let $Y$ be a smooth projective surface and let $|V|$ be a basepoint-free linear system on $Y.$  Denote by $\mathfrak{U}$ the open subset of $|V|$ parametrizing smooth members of $|V|,$ and let $$\mathcal{V}:=\{([C],p) \in \mathfrak{U} \times Y : p \in C \}$$ be the associated incidence variety.

Let $\mathfrak{M}$ be a vector bundle of rank $s$ on $Y$ such that for some $[C_{0}] \in \mathfrak{U},$ the restriction $\mathfrak{M}|_{C_{0}}$ admits a theta-divisor in $\textnormal{Pic}^{d}(C_{0}).$  Then there exists a Zariski-open subset $\widetilde{\mathfrak{U}} \subseteq \mathfrak{U}$ containing $[C_{0}]$ such that for all $[C] \in \widetilde{\mathfrak{U}}$ the vector bundle $\mathfrak{M}|_{C}$ admits a theta-divisor in $\textnormal{Pic}^{d}(C).$
\end{prop}

\begin{proof}
Let $g$ be the arithmetic genus of an element of $|V|.$  Assume without loss of generality that $d \geq g,$ so that for each $[C] \in \mathfrak{U}$ every line bundle of degree $d$ on $C$ has a section.  Fix a divisor $D$ on $Y$ which is sufficiently ample to guarantee the vanishings $H^{1}(\mathfrak{M}(D))=0$ and $H^{2}(\mathfrak{M}(D-C))=0$ for all $[C] \in \mathfrak{U}.$  (Note that we have listed only two vanishings, since the elements of $\mathfrak{U}$ are linearly equivalent to one another.)  This implies that
\begin{equation}
\label{allvanish}
H^{1}(\mathfrak{M}(D)|_{C})=0\textnormal{ for all }[C] \in \mathfrak{U}.
\end{equation}
Fix $C \in \mathfrak{U}$ and $\mathcal{L} \in \textnormal{Pic}^{d}(C).$  Then $H^{0}(\mathcal{L}) \neq 0$ by our assumption on $d,$ so for some nonzero $\mathfrak{s} \in H^{0}(\mathcal{L})$ we have an exact sequence
\begin{equation}
\label{minisequence}
0 \rightarrow \mathcal{O}_{C} \overset{\mathfrak{s}}{\rightarrow} \mathcal{L} \rightarrow \mathcal{O}_{\mathfrak{s}^{-1}(0)} \rightarrow 0
\end{equation}
Twisting (\ref{minisequence}) by $\mathfrak{M}(D)|_{C}$ and taking cohomology, we now have from (\ref{allvanish}) that $H^{1}(\mathfrak{M}(D)|_{C} \otimes \mathcal{L})=0.$  Furthermore, our hypothesis that the restriction of $\mathfrak{M}$ to the genus-$g$ curve $C_{0}$ in $|V|$ admits a theta-divisor in $\textnormal{Pic}^{d}(C_{0})$ implies that $c_{1}(\mathfrak{M}|_{C} \otimes \mathcal{L})=s(g-1)$ for some positive integer $s.$  We may then conclude from Riemann-Roch that
\begin{equation}
\label{constrank}
H^{0}(\mathfrak{M}(D)|_{C} \otimes \mathcal{L})=s(C \cdot D)=s(C_{0} \cdot D) \hskip5pt {\forall}[C] \in \mathfrak{U},{\forall}\mathcal{L} \in \textnormal{Pic}^{d}(C).
\end{equation}
To construct our open set $\widetilde{\mathfrak{U}} \subseteq \mathfrak{U},$ we consider the Cartesian diagram
\begin{equation}
\label{cartdiag}
\begin{small}
\xymatrix{
{\mathcal{V}' \times_{\mathfrak{U}'} \textnormal{\textbf{Pic}}^{d}_{\mathfrak{U}'}(\mathcal{V}')} \ar[r]^-{h'} \ar[d]^-{{\pi}''} &{\mathcal{V'}:=\mathfrak{U}' \times_{\mathfrak{U}} \mathcal{V}} \ar[r]^-{h} \ar[d]^-{{\pi}'} &\mathcal{V} \ar[r]^-{\eta} \ar[d]^-{\pi} &Y\\
{\textnormal{\textbf{Pic}}^{d}_{\mathfrak{U}'}(\mathcal{V}')} \ar[r]^-{g'} &\mathfrak{U}' \ar[r]^-{g} &\mathfrak{U}
}
\end{small}
\end{equation}
where $g: \mathfrak{U}' \rightarrow \mathfrak{U}$ is a finite base change for which ${\pi}': \mathcal{V}' \rightarrow \mathfrak{U}'$ admits a section ${\sigma}':\mathfrak{U}' \rightarrow \mathcal{V}'$ and $\textnormal{\textbf{Pic}}^{d}_{\mathfrak{U}'}(\mathcal{V}')$ is the relative Picard variety of degree $d.$  The existence of ${\sigma}'$ implies the existence of a Poincar\'{e} line bundle on $\mathcal{V}' \times_{\mathfrak{U}'} \textnormal{\textbf{Pic}}^{d}_{\mathfrak{U}'}(\mathcal{V}')$, i.e a line bundle $\mathcal{P}'$ such that for all $u' \in \mathfrak{U}'$ and all line bundles $\mathcal{L}'$ of degree $d$ on the curve $C':=(\pi')^{-1}(u'),$ the restriction of $\mathcal{P}'$ to $C' \times [\mathcal{L}']$ is isomorphic to $\mathcal{L'}.$

We define two coherent sheaves on $\textnormal{\textbf{Pic}}^{d}_{\mathfrak{U}'}(\mathcal{V}')$ as follows:
$$\mathcal{F}:=({\pi}'')_{\ast}((\eta \circ h \circ h')^{\ast}(\mathfrak{M}(D)) \otimes \mathcal{P}')$$ $$\mathcal{G}:=({\pi}'')_{\ast}((\eta \circ h \circ h')^{\ast}(\mathfrak{M}(D)) \otimes \mathcal{P}' \otimes \mathcal{O}_{(\eta \circ h \circ h')^{-1}(D)})$$

The fiber of $\mathcal{F}$ (resp. $\mathcal{G}$) over $(C',\mathcal{L}') \in \textnormal{\textbf{Pic}}^{d}_{\mathfrak{U}'}(\mathcal{V}')$ is $H^{0}(\mathfrak{M}(D)|_{C'} \otimes \mathcal{L}')$ (resp. $H^{0}(\mathfrak{M}(D)|_{C' \cap D} \otimes \mathcal{L}|_{C' \cap D})$), where $C' \cap D$ denotes the scheme-theoretic intersection of $C'$ and $D.$  We then have from (\ref{constrank}) and Grauert's Theorem that $\mathcal{F}$ and $\mathcal{G}$ are both vector bundles of rank $s(C_{0} \cdot D)$ on $\textbf{Pic}_{\mathfrak{U}'}^{d}(\mathcal{V}').$  Moreover, there is a natural morphism of vector bundles $\rho:\mathcal{F} \rightarrow \mathcal{G}$ whose fiber at each $(C',\mathcal{L}')$ is the restriction map  $H^{0}(\mathfrak{M}(D)|_{C'} \otimes \mathcal{L}') \rightarrow H^{0}(\mathfrak{M}(D)|_{C' \cap D} \otimes \mathcal{L}|_{C' \cap D})$ with kernel $H^{0}(\mathfrak{M}|_{C'} \otimes \mathcal{L}').$  The zero locus of the determinant of $\rho$ will be denoted by $\mathcal{D}_{\rho},$ and it is supported on the set
\begin{equation}
\textnormal{supp}(\mathcal{D}_{\rho}):=\{(C',\mathcal{L}') \in \textbf{Pic}^{d}_{\mathfrak{U}'}(\mathcal{V}') : H^{0}(\mathfrak{M}|_{C'} \otimes \mathcal{L}') \neq 0\}.
\end{equation}
Our next task is to show that $\mathcal{D}_{\rho}$ is of pure codimension 1 in $\textbf{Pic}^{d}_{\mathfrak{U}'}(\mathcal{V}').$  Since $\mathcal{D}_{\rho}$ is locally the zero locus of a single function, it suffices to check that $\textnormal{supp}(\mathcal{D}_{\rho})$ is a nonempty proper subset of $\textbf{Pic}^{d}_{\mathfrak{U}'}(\mathcal{V}').$

Choose a point $u'_{0} \in \mathfrak{U}'$ for which $h(({\pi}')^{-1}(u'_{0}))=\pi^{-1}(g(u'_{0}))=C_{0}.$  By hypothesis, the intersection of $\mathfrak{D}_{\rho}$ with $(g')^{-1}(u'_{0}) \cong \textnormal{Pic}^{d}(C_{0})$ is an effective divisor in $\textnormal{Pic}^{d}(C_{0});$ therefore $\textnormal{supp}(\mathfrak{D}_{\rho})$ is a nonempty proper subset of $\textbf{Pic}^{d}_{\mathfrak{U}'}(\mathcal{V}')$ as claimed.

Since the fiber of the restriction $g'|_{\mathcal{D}_{\rho}}: \mathcal{D}_{\rho} \rightarrow \mathfrak{U}'$ over $u'_{0}$ is of dimension $g-1,$ and we have just seen that $\dim{\mathcal{D}}_{\rho}=g-1+\dim{\mathfrak{U}'},$ semicontinuity implies that there exists a nonempty Zariski-open subset $\widetilde{\mathfrak{U}'} \subseteq \mathfrak{U}'$ containing $u'_{0}$ over which $g'|_{\mathcal{D}_{\rho}}$ has relative dimension $g-1.$

Finally, we define $\widetilde{\mathfrak{U}}$ to be the Zariski interior of the image set $g(\widetilde{\mathfrak{U}'}).$  It is clear from the properties of $\widetilde{\mathfrak{U}'}$ that $\widetilde{\mathfrak{U}}$ contains $[C_{0}]$ and that the restriction of $\mathfrak{M}$ to $C$ admits a theta-divisor in $\textnormal{Pic}^{d}(C)$for all $[C] \in \widetilde{\mathfrak{U}}.$
\end{proof}

\begin{thm}
\label{mainprop1}
Let $\mathcal{E}$ be an Ulrich bundle of rank $r$ on $X_d.$  Then for a general smooth member $C$ of $|\det{\mathcal{E}}|,$ we have that $M_{C}$ admits a theta-divisor in $\textnormal{Pic}^{c_{2}(\mathcal{E})}(C).$
\end{thm}

\begin{proof}
According to Proposition \ref{thetaopen}, it suffices to exhibit just one smooth member of $|\det{\mathcal{E}}|$ for which the conclusion holds.  Since $\mathcal{E}$ is Ulrich, it is globally generated with $dr$ global sections, so a general choice of $r$ global sections yields an injective morphism $\sigma: \mathcal{E}^{\vee} \rightarrow \mathcal{O}_X^{r}$ whose cokernel is a line bundle $\mathcal{L}$ on a smooth curve $C \subseteq X$ of degree $dr$ and genus $c_{2}(\mathcal{E})-r+1=\frac{C^{2}-dr}{2}+1.$  We therefore have an exact sequence
\begin{equation}
\label{propseq1}
0 \longrightarrow \mathcal{E}^{\vee} \overset{\sigma}{\longrightarrow} \mathcal{O}_X^{r} \longrightarrow j_{\ast}\mathcal{L} \longrightarrow 0
\end{equation}
where $j:C \hookrightarrow X$ is inclusion.  Since $\mathcal{L}$ has the desired degree $c_{2}(\mathcal{E})=\frac{C^{2}-(d-2)r}{2},$ it remains to check that $H^{0}(M_{C} \otimes \mathcal{L})=0.$

If we define $M_{X}:=\ker(H^{0}(\mathcal{O}_{X}(1)) \otimes \mathcal{O}_{X} \rightarrow \mathcal{O}_{X}(1)),$ then $M_{X} \otimes \mathcal{O}_{C} \cong M_{C},$ and tensoring (\ref{propseq1}) with the sequence
\begin{equation}
\label{propseq3}
0 \rightarrow M_{X} \rightarrow H^{0}(\mathcal{O}_{X}(1)) \otimes \mathcal{O}_{X} \rightarrow \mathcal{O}_{X}(1) \rightarrow 0
\end{equation}
yields the following commutative diagram with exact rows and columns:

\medskip
\begin{center}
\begin{tiny}
\label{propbig}
$\xymatrix{
            &0                                                                                  &0                                                               &0                                                         & \\
0 \ar[r] &{j_{\ast}(M_{C} \otimes \mathcal{L})} \ar[r] \ar[u] &{H^{0}(\mathcal{O}_{X}(1)) \otimes j_{\ast}\mathcal{L}} \ar[r] \ar[u] &{j_{\ast}\mathcal{L}(1)} \ar[r] \ar[u] &0 \\
0 \ar[r] &{M_{X}^{r}} \ar[r] \ar[u] &{H^{0}(\mathcal{O}_{X}(1)) \otimes \mathcal{O}_{X}^{r}} \ar[r] \ar[u] &{\mathcal{O}_{X}(1)^{r}} \ar[r] \ar[u] &0 \\
0 \ar[r] &{M_{X} \otimes \mathcal{E}^{\vee}} \ar[r] \ar[u] &{H^{0}(\mathcal{O}_{X}(1)) \otimes \mathcal{E}^{\vee}} \ar[r] \ar[u] &\mathcal{E}^{\vee}(1) \ar[r] \ar[u] &0 \\
          &0 \ar[u] &0 \ar[u] &0 \ar[u] & \\
}$
\end{tiny}
\end{center}
\medskip
It follows from taking cohomology in (\ref{propseq3}) that $h^{0}(M_{X})=h^{1}(M_{X})=0.$  The fact that Ulrich bundles are ACM and normalized yields the vanishings $h^i(\mathcal{E}^{\vee})=h^{2-i}(\mathcal{E}(-1))=0$ for $i=0,1$ and $h^1(\mathcal{E}^{\vee})=0.$  Furthermore, $H^0(\mathcal{E}^{\vee}(1)) \cong H^2(\mathcal{E}(-2))^{\ast}$ by Serre duality, and the Ulrich condition implies the vanishing of $H^2(\mathcal{E}(-2)),$ so we also have $h^0(\mathcal{E}(-1))=0.$  Applying cohomology to the previous diagram then yields the following commutative diagram with exact rows and columns:

\medskip
\begin{center}
\begin{tiny}
$\xymatrix{
            &                                                                      &0                                                                                                                &0                                                       &  \\
0 \ar[r] &H^{0}(M_{C} \otimes \mathcal{L}) \ar[r] &H^{0}(\mathcal{O}_{X}(1)) \otimes H^{0}(\mathcal{L}) \ar[r] \ar[u] &H^{0}(\mathcal{L}(1))  \ar[u] & \\
            &0 \ar[r]                                                          &H^{0}(\mathcal{O}_{X}(1)) \otimes H^{0}(\mathcal{O}_{X})^{r} \ar[r] \ar[u] &H^{0}(\mathcal{O}_{X}(1))^{r} \ar[r] \ar[u] &0 \\
            &                                                                       &0 \ar[u]                                                                                                     &0 \ar[u]                                             & \\
}$
\end{tiny}
\end{center}
\medskip
It is immediate from this diagram that the map $H^{0}(\mathcal{O}_{X}(1)) \otimes H^{0}(\mathcal{L}) \rightarrow H^{0}(\mathcal{L}(1))$ is an isomorphism, which implies in turn that $H^{0}(M_{C} \otimes \mathcal{L})=0.$
\end{proof}

We now start towards establishing the implication $(ii) \Rightarrow (i).$  First, we need a lemma.

\begin{lem}
\label{bpfln}
Let $C \subseteq \mathbb{P}^{n}$ be a smooth irreducible projective curve of genus $g$ and degree $dr$ for some $d,r \geq 2.$  Then the general line bundle $\mathcal{L}$ of degree $g-1+r$ on $C$ is nonspecial, basepoint-free, and satisfies $h^{0}(\mathcal{L}(-1))=0.$
\end{lem}

\begin{proof}
By geometric Riemann-Roch, the nonspecial line bundles of degree $g-1+r$ form a nonempty Zariski-open subset of $\textnormal{Pic}^{g-1+r}(C).$  We now show that the same is true for basepoint-free line bundles of degree $g-1+r.$  Consider the Brill-Noether locus
\begin{equation}
\label{bnloc}
W^{r-1}_{g-2+r}(C)=\{\mathcal{L'} \in \textnormal{Pic}^{g-2+r}(C): h^{0}(\mathcal{L}') \geq r\}
\end{equation}
and the natural map
\begin{equation}
\phi: C \times W^{r-1}_{g-2+r}(C) \rightarrow \textnormal{Pic}^{g-1+r}(C), \hskip10pt (p,\mathcal{L}') \mapsto \mathcal{L}'(p)
\end{equation}
It is clear that the image of $\phi$ contains the set of all line bundles of degree $g-1+r$ which possess a base point.  By Martens' Theorem (p.191-192 of \cite{ACGH}) the dimension of $W^{r-1}_{g-2+r}(C)$ is at most $g-r,$ so that the image of $\phi$ has dimension at most $g-r+1.$  Since $r \geq 2,$ the complement of the image of $\phi$ is nonempty and Zariski-open.

To see that the vanishing of $h^{0}(\mathcal{L}(-1))$ is a nonempty and Zariski-open condition on $\mathcal{L}\in \textnormal{Pic}^{g-1+r}(C),$ recall that the locus in $\textnormal{Pic}^{g-1-(d-1)r}(C)$ parametrizing line bundles with a nonzero section has dimension equal to $g-1-(d-1)r$ by Abel's Theorem.
\end{proof}

The next result concludes the proof of Theorem \ref{main}.

\begin{thm}
\label{mainprop2}
Let $C \subseteq X_d$ be a smooth irreducible curve of degree $dr \geq d$ and genus $g,$ and assume that $M_{C}$ admits a theta-divisor in $\textnormal{Pic}^{g-1+r}(C).$  Then there exists an Ulrich bundle $\mathcal{E}$ of rank $r$ on $X$ with $c_{1}(\mathcal{E})=C$ and $c_{2}(\mathcal{E})=g-1+r.$
\end{thm}

\begin{proof}
By the adjunction formula, we have that $g-1+r=\frac{C^{2}-(d-2)r}{2}.$  Since the vanishing of $H^{0}(M_{C} \otimes \mathcal{L})$ is an open condition on $\mathcal{L},$ it follows from our hypothesis and Lemma \ref{bpfln} that there exists a nonspecial and basepoint-free line bundle $\mathcal{L}$ of degree $\frac{C^{2}-(d-2)r}{2}$ on $C$ for which $h^{0}(\mathcal{L}(-1))=0$ and $h^{0}(M_{C} \otimes \mathcal{L})=0.$  We fix such a line bundle for the rest of the proof.

We compose the evaluation map $H^{0}(\mathcal{L}) \otimes \mathcal{O}_{C} \rightarrow \mathcal{L}$ with the restriction map $H^{0}(\mathcal{L}) \otimes \mathcal{O}_{X} \rightarrow H^{0}(\mathcal{L}) \otimes \mathcal{O}_{C}$ to produce a surjection $$\rho: H^{0}(\mathcal{L}) \otimes \mathcal{O}_{X} \rightarrow j_{\ast}\mathcal{L}.$$  (As before, $j$ denotes the inclusion $C \hookrightarrow X.$)

Since $\mathcal{L}$ is nonspecial, we have that $h^{0}(\mathcal{L})=r,$ and it follows that $\ker(\rho)$ is a torsion-free sheaf of rank $r$ on $X.$  Observe that dualizing the sequence
\begin{equation}
\label{deff}
0 \longrightarrow \ker(\rho) \longrightarrow H^{0}(\mathcal{L}) \otimes \mathcal{O}_{X} \overset{\rho}{\longrightarrow} j_{\ast}\mathcal{L} \longrightarrow 0
\end{equation}
yields an isomorphism $\mathcal{E}xt^{1}_{\mathcal{O}_{X}}(\ker(\rho),\mathcal{O}_{X}) \cong \mathcal{E}xt^{2}_{\mathcal{O}_{X}}(j_{\ast}\mathcal{L},\mathcal{O}_{X}).$  Since $j_{\ast}\mathcal{L}$ is supported on a codimension-1 subvariety of $X,$ it follows that $\mathcal{E}xt^{1}_{\mathcal{O}_{X}}(\ker(\rho),\mathcal{O}_{X})=0;$ consequently the torsion-free sheaf $\ker(\rho)$ is locally free.  We define $\mathcal{E}:=\ker(\rho)^{\vee}.$  

We will show that $\mathcal{E}$ is Ulrich.  A straightforward Chern class calculation (e.g. Proposition 5.2.2 in \cite{HL}) applied to (\ref{deff}) shows that $c_1(\mathcal{E})=C$ and $c_2(\mathcal{E})=\frac{C^2-(d-2)r}{2},$ so by Corollary \ref{chern} it suffices to verify that $\mathcal{E}$ is ACM.  Given that $H^1(\mathcal{E}^{\vee}(t)) \cong H^1(\mathcal{E}(-t-1))^{\ast}$ for all $t \in \mathbb{Z}$ by Serre duality, it suffices in turn to show that $\mathcal{E}^{\vee}$ is ACM. 

Twisting (\ref{deff}) by $t \in \mathbb{Z}$ yields the sequence
\begin{equation}
\label{defftwist}
0 \longrightarrow \mathcal{E}^{\vee}(t) \longrightarrow H^{0}(\mathcal{L}) \otimes \mathcal{O}_{X}(t) {\longrightarrow} j_{\ast}\mathcal{L}(t) \longrightarrow 0
\end{equation}


Taking cohomology in (\ref{defftwist}) yields the exact sequence

\begin{small}
\begin{equation*}
H^{0}(\mathcal{L}) \otimes H^{0}(\mathcal{O}_{X}(t)) \overset{\mu_{t}}{\longrightarrow} H^{0}(\mathcal{L}(t)) \longrightarrow H^{1}(\mathcal{E}^{\vee}(t)) \longrightarrow H^{0}(\mathcal{L}) \otimes H^{1}(\mathcal{O}_{X}(t))
\end{equation*}
\end{small}

where $\mu_{t}$ is multiplication of sections.  Since $\mathcal{O}_{X}$ is ACM, $H^{1}(\mathcal{O}_{X}(t))=0,$ and it follows that
\begin{equation}
\label{cokerchar}
H^{1}(\mathcal{E}^{\vee}(t))=\textnormal{coker}(H^{0}(\mathcal{L}) \otimes H^{0}(\mathcal{O}_{X}(t)) \overset{\mu_{t}}{\longrightarrow} H^{0}(\mathcal{L}(t))).
\end{equation}
We will use this characterization of $H^{1}(\mathcal{E}^{\vee}(t))$ to show that $\mathcal{E}^{\vee}$ is ACM.  There are three cases to consider.

\medskip

\textbf{Case I:  $(t \geq 1)$} We proceed by induction on $t.$  Since $H^{0}(M_{C} \otimes \mathcal{L})=0$ by hypothesis and $\chi(M_{C} \otimes \mathcal{L})=0$ by Riemann-Roch, we have $H^{1}(M_{C} \otimes \mathcal{L})=0;$ this will play an important role in what follows.

The base case $t=1$ is established by observing that the vanishing of $H^{i}(M_{C} \otimes \mathcal{L})$ for $i=0,1$ is equivalent to $\mu_{1}$ being an isomorphism.

Assume now that $t \geq 2$ and that $H^{1}(\mathcal{E}^{\vee}(t-1))=0.$  We have the following commutative diagram consisting of multiplication maps:

\medskip
\begin{center}
\begin{tiny}
$\xymatrix{
H^{0}(\mathcal{O}_{X}(t-1)) \otimes H^{0}(\mathcal{O}_{X}(1)) \otimes H^{0}(\mathcal{L}) \ar[r]^-{\mu_{t-1} \otimes \textnormal{id}} \ar[d]^-{\nu_{t} \otimes \textnormal{id}} &H^{0}(\mathcal{L}(t-1)) \otimes H^{0}(\mathcal{O}_{X}(1)) \ar[d]^-{\xi_{t}} \\
H^{0}(\mathcal{O}_{X}(t)) \otimes H^{0}(\mathcal{L}) \ar[r]^-{\mu_{t}} &H^{0}(\mathcal{L}(t)) \\
}$
\end{tiny}
\end{center}
\medskip

Since $\mu_{t-1} \otimes \textnormal{id}$ is surjective by our inductive hypothesis, we will know that $H^{1}(\mathcal{E}^{\vee}(t))=0$ once we check that $\xi_{t}$ is surjective.  Given that $\mathcal{L}$ is nonspecial and $t \geq 2,$ it follows that $\mathcal{L}(t-1)$ is nonspecial, and one easily sees that
\begin{equation}
\textnormal{coker}(\xi_{t}) \cong H^{1}(M_{C} \otimes \mathcal{L}(t-1)).
\end{equation}
If we fix an effective divisor $D \in |\mathcal{O}_{C}(t-1)|,$ we obtain an exact sequence
\begin{equation}
0 \rightarrow M_{C} \otimes \mathcal{L} \rightarrow M_{C} \otimes \mathcal{L}(t-1) \rightarrow (M_{C} \otimes \mathcal{L}(t-1))|_{D} \rightarrow 0
\end{equation}
Taking cohomology then shows that $H^{1}(M_{C} \otimes \mathcal{L}(t-1))$ is a quotient of $H^{1}(M_{C} \otimes \mathcal{L}).$  Since the latter is assumed to be 0, we are done.

\medskip

\textbf{Case II:  $(t=0)$} Immediate.

\medskip

\textbf{Case III: $(t \leq -1)$} Since $H^{0}(\mathcal{L}(-1))=0,$ it follows that $H^{0}(\mathcal{L}(t))=0,$ which implies via (\ref{cokerchar}) that $H^{1}(\mathcal{E}^{\vee}(t))=0.$
\end{proof}

\begin{rem}
If $C$ is a smooth irreducible curve on a smooth AG surface $X \subseteq \mathbb{P}^{n}$ of degree $d$ satisfying $K_{X}=mH,$ and the degree of $C$ is $(\frac{m+3}{2})dr,$ then the bundle $M_{C}$ has slope $-(\frac{m+3}{2n})dr.$  In the case where $X$ is a quadric surface, this slope is $-\frac{r}{3},$ so it is impossible for $M_{C}$ to admit a theta-divisor in ${\rm Pic}^{g-1+r}(C);$ this gives an alternate explanation of why Theorem \ref{main} does not extend to the quadric surface case.
\end{rem}

The remainder of this section is devoted to the construction of Ulrich bundles on $X_d$ of rank $rH$ for $3 \leq d \leq 7$ and $r \geq 2.$  While this has already been done in the papers \cite{CH},\cite{MP2} and \cite{CH1}, we present an alternate construction based on Theorem \ref{main}.  

\begin{prop}
\label{hplusrnc}
Suppose $3 \leq d \leq 7.$  If $Q$ is the class of a rational normal curve of degree $d$ on $X_d,$ then there exists a rank-2 Ulrich bundle $\mathcal{E}$ on $X_d$ with $c_{1}(\mathcal{E})=H+Q.$
\end{prop}

\begin{proof}
Let $C \in |H+Q|$ be a smooth irreducible curve.  It is straightforward to check that $C$ is of genus $d$ and is embedded in $\mathbb{P}^{d}$ as a curve of degree $2d$ by the complete linear series $|\mathcal{O}_{C}(1)|.$  By Theorem \ref{mainprop2}, it suffices to show that for the general line bundle $\mathcal{L}$ of degree $d+1$ on $C,$ the multiplication map $\mu:H^{0}(\mathcal{L}) \otimes H^{0}(\mathcal{O}_{C}(1)) \rightarrow H^{0}(\mathcal{L}(1))$ is an isomorphism.

Since the Brill-Noether loci $W^{2}_{d+1}(C)$ and $W^{1}_{d}(C)$ are proper subvarieties of $\textnormal{Pic}^{d+1}(C)$ and $\textnormal{Pic}^{d}(C),$ respectively, the complete linear series determined by a general line bundle of degree $d+1$ on $C$ is a basepoint-free pencil.  Let $\mathcal{L}$ be such a line bundle.  By the basepoint-free pencil trick, we have
\begin{equation}
0 \rightarrow \mathcal{L}^{-1} \rightarrow H^{0}(\mathcal{L}) \otimes \mathcal{O}_{C} \rightarrow \mathcal{L} \rightarrow 0
\end{equation}
Twisting by 1 and taking cohomology, we see that the kernel of $\mu$ is isomorphic to $H^{0}(\mathcal{L}^{-1}(1)).$  Given that $\mathcal{L}$ is general of degree $d+1,$ the twist $\mathcal{L}^{-1}(1)$ is a general line bundle of degree $d-1.$  Since $W_{d-1}(C)$ is a proper subvariety (indeed, an effective divisor) of $\textnormal{Pic}^{d-1}(C),$ we have that $H^{0}(\mathcal{L}^{-1}(1))=0.$
\end{proof}

\begin{prop}
\label{hypclass}
Suppose $3 \leq d \leq 7.$  For each $r \geq 2$ there exists an Ulrich bundle of rank $r$ on $X_d$ with first Chern class $rH.$
\end{prop}

\begin{proof}
It suffices to check the cases $r=2$ and $r=3,$ since the remaining cases can be treated by taking direct sums.  Fix a rational normal curve class $Q.$  Then $2H-Q$ is also a rational normal curve class; consequently $\mathcal{O}_{X}(Q) \oplus \mathcal{O}_{X}(2H-Q)$ is a rank-2 Ulrich bundle with first Chern class $2H.$

Turning to the rank-3 case, we have from Proposition \ref{hplusrnc} there exists a rank-2 Ulrich bundle $\mathcal{F}$ with $c_{1}(\mathcal{F})=H+Q.$  It follows that $\mathcal{E}':=\mathcal{O}_{X}(2H-Q) \oplus \mathcal{F}$ is an Ulrich bundle of rank 3 with $c_{1}(\mathcal{E}')=3H.$ 
\end{proof}

\begin{rem} The existence of rank-2 Ulrich bundles on $X_d$ with first Chern class $2H$ was deduced in Corollary 6.5 of \cite{ESW}.  The proof of this corollary, which may be considered a precursor to our methods, involves an application of the basepoint-free pencil trick similar to the proof of Proposition \ref{hplusrnc}.
\end{rem}

\section{The Minimal Resolution Conjecture}
\label{relminres}

In this final section we give the proofs of Corollary \ref{minrescor} and Theorem \ref{mrchyp}. Our account of the Minimal Resolution Conjecture (MRC) will be very brief; for further details, see \cite{FMP}, \cite{Cas} and the references therein.

Let $Y \subseteq \mathbb{P}^{N}=\mathbb{P}(V)$ be a proper closed subscheme, let $S=\textnormal{Sym}(V),$ and let $P_{Y}$ be the Hilbert polynomial of $Y.$  Denote by $b_{i,j}(Y)$ the Betti numbers associated to the minimal free resolution

\begin{equation}
\label{basicbetti}
0 \leftarrow I_{Y/\mathbb{P}^{N}} \leftarrow \bigoplus_{j_{1}=1}^{l}S(-j_{1})^{b_{1,j_{1}}(Y)} \leftarrow  \cdots \leftarrow \bigoplus_{j_{k}=1}^{l}S(-j_{k})^{b_{k,j_{k}}(Y)} \leftarrow 0
\end{equation}
of the saturated homogeneous ideal $I_{Y/\mathbb{P}^{N}}.$  The integer $l+1$ is understood to be the Castelnuovo-Mumford regularity
\begin{equation}
\textnormal{reg}(I_{Y/\mathbb{P}^{N}})=\max\{m: b_{s,m-1}(Y) \neq 0 \textnormal{ for some }s\}.
\end{equation}
The resolution (\ref{basicbetti}) may be encoded by an array as follows:
\medskip
\begin{center}
\begin{tiny}
$\begin{array}{ccccc}
1 & - & \cdots & - & - \\
-  & b_{1,1}(Y) & \cdots & b_{k-1,1}(Y) & b_{k,1}(Y) \\
-  & b_{1,2}(Y) & \cdots & b_{k-1,2}(Y) & b_{k,2}(Y) \\
\vdots & \vdots & \vdots & \vdots & \vdots \\
- & b_{1,l}(Y) & \cdots & b_{k-1,l}(Y) & b_{k,l}(Y)
\end{array}$
\end{tiny}
\end{center}
\medskip
This array is the \textit{Betti diagram} of $Y.$  (The reader should be aware that other indexing conventions are in wide use.)

\begin{defn}
Let $\Gamma$ be a zero-dimensional subscheme of $Y.$  We say that $\Gamma$ \textit{satisfies the MRC for $Y$} if
\begin{equation}
b_{i+1,q-1}(\Gamma) \cdot b_{i,q}(\Gamma)=0 \hskip0.5cm \textnormal{for all }i
\end{equation}
whenever $q \geq \textnormal{reg}(I_{Y/\mathbb{P}^{N}})+1.$
\end{defn}

Corollary \ref{minrescor} follows from Theorem \ref{main} and a special case of Corollary 1.8 in \cite{FMP} which we now state without proof.  (Compare Remark 1.10 in \textit{loc.~cit.})

\begin{prop}
\label{fmpmrc}
Let $C \subseteq \mathbb{P}^{n}$ be a smooth irreducible curve of genus $g$ and degree $d$ and let $P_{C}(t)$ be the Hilbert polynomial of $C.$  Assume further that $n|d.$  Then every collection $\Gamma$ of $\gamma \geq \max\{g,P_{C}(\textnormal{reg}(I_{C|\mathbb{P}(V)})\}$ general points on $C$ satisfies $MRC$ if and only if for every $i \leq \frac{n}{2}$ and a general line bundle $\xi \in \textnormal{Pic}^{g-1+\frac{di}{n}}(C)$ we have that $H^{0}(\wedge^{i}M_{C} \otimes \xi)=0.$  \hfill \qedsymbol
\end{prop}




We now illustrate Corollary \ref{minrescor} with an example.  Let $C$ be a smooth member of the linear system $|5{\ell}-4e_{1}-e_{2}-e_{3}|$ on the smooth cubic surface $X_3,$ where $\ell$ is the pullback of the hyperplane class via a blowdown $p : X \rightarrow \mathbb{P}^{2}$ of the mutually disjoint (-1)-curves $e_1, \cdots ,e_6.$  This is a smooth rational curve of degree $9$ on $X_3;$ however, since $C^2 < 3$ it follows from Proposition \ref{csquared} that $C$ cannot represent the first Chern class of any Ulrich bundle of rank 3 on $X_3.$  The minimal graded resolution of $C \subseteq \mathbb{P}^{3}$ has the Betti diagram
\medskip
\begin{center}
\begin{tiny}
$\begin{array}{cccc}
1 & - & - & -  \\
-  & - & - & -  \\
-  &1 & - & -  \\
-  &- & - & - \\
- & 3 & 3 & - \\
- & 1 & 2 & 1 \\
- & 1 & 2 & 1 \\
- & 1 & 2 & 1
\end{array}$
\end{tiny}
\end{center}
\medskip
Since $P_{C}(t)=9t+1$ and $\textnormal{reg}(I_{C/\mathbb{P}^{3}})=8,$ we have that $P_{C}(\textnormal{reg}(I_{C/\mathbb{P}^{3}}))=73.$  If $\Gamma \subseteq C$ is a collection of 75 general points on $C,$ then the Betti diagram of $\Gamma$ is
\medskip
\begin{center}
\begin{tiny}
$\begin{array}{cccc}
1 & - & - & -  \\
-  & - & - & -  \\
-  &1 & - & -  \\
-  &- & - & - \\
- & 3 & 3 & - \\
- & 1 & 2 & 1 \\
- & 1 & 2 & 1 \\
- & 1 & 2 & 1 \\
- & 7 & 12 & 4 \\
- & - & 1 & 2
\end{array}$
\end{tiny}
\end{center}
\medskip
Note that this differs from the Betti diagram of $C$ only in the last two rows.  Given that $b_{3,8}(\Gamma)=4$ and $b_{2,9}(\Gamma)=1,$ the set $\Gamma$ fails to satisfy MRC for $C.$

\begin{lem}
\label{ratnorm}
Let $Q \subseteq X_d$ be a rational normal curve of degree $d.$  Then $Q$ satisfies MRC.
\end{lem}

\begin{proof}
Since $\mathcal{O}_{X_d}(Q)$ is an Ulrich line bundle, we have from Corollary \ref{curvess} that $M_Q$ is a semistable vector bundle on $Q \cong \mathbb{P}^1$ of rank $d$ and degree $-d.$  Grothendieck's Theorem combined with semistability implies that $M_{Q} \cong \mathcal{O}_{\mathbb{P}^1}(-1)^{\oplus d},$ so that $\wedge^{i}M_{Q} \cong \mathcal{O}_{\mathbb{P}^1}(-i)^{\oplus \binom{d}{i}}.$  Consequently, for $j=0,1$ we have $$H^{j}(\wedge^{i}M_{Q} \otimes \mathcal{O}_{\mathbb{P}^1}(i-1))=H^{j}(\mathcal{O}_{\mathbb{P}^1}(-1))^{\oplus \binom{d}{i}}=0.$$  The desired result then follows from Proposition \ref{fmpmrc}.
\end{proof}

We now turn to the proof of Theorem \ref{mrchyp}.  Recall from the Introduction that the Ulrich semigroup $\mathfrak{Ulr}(X_d)$ of ${\rm Pic}(X_d)$ is the set of divisor classes which are first Chern classes of Ulrich bundles on $X_d.$ 

\begin{prop}
\label{mrcsum}
Let $D_1, \cdots ,D_m \in \mathfrak{Ulr}(X_d)$ satisfy the property that for each $j=1, \cdots ,m$ the general smooth member $C_j$ of $|D_j|$ satisfies the MRC.  Then the general smooth member $C$ of $|D_1 + \cdots + D_m|$ satisfies the MRC.
\end{prop}

\begin{proof} 
It suffices to handle the case $m=2;$ the general case will then follow by induction.

For $j=1,2,$ let $dr_{j}$ and $g_j$ be the degree and genus, respectively, of each member of the linear system $|D_j|,$ and define $r:=r_{1}+r_{2}.$  Fix $i \leq \frac{d}{2}.$  By our hypothesis together with Propositions \ref{thetaopen} and \ref{fmpmrc}, we may choose for $j=1,2$ a smooth curve $C_j \in |D_j|$ satisfying the following properties:
\begin{itemize}
\item[(i)]{$C_{1}+C_{2} \in |D_{1}+D_{2}|.$}
\item[(ii)]{$C_{1}$ meets $C_{2}$ transversally in $m:=C_1 \cdot C_2$ points.}
\item[(iii)]{$\wedge^{i}M_{C_j}$ admits a theta-divisor in $\textnormal{Pic}^{g_{j}-1+r_{j}i}(D).$}
\end{itemize}
Any member of the linear system $|D_{1}+D_{2}|$ has arithmetic genus equal to $g:=g_{1}+g_{2}-m+1.$  We will construct a pencil of smooth curves in $|D_{1}+D_{2}|$ degenerating to the nodal curve $C_1+C_2$ and show that for a general smooth member $C$ of this pencil, the vector bundle $\wedge^{i}M_{C}$ admits a theta-divisor in $\textnormal{Pic}^{g-1+ri}(C).$     

Since $D_1+D_2$ is the first Chern class of an Ulrich bundle of rank 2 or greater, the linear system $|D_{1}+D_{2}|$ is basepoint-free of dimension at least 2, so we may choose a sub-pencil $f:X_d \dashrightarrow \mathbb{P}^{1}$ of $|D_{1}+D_{2}|$ with $\overline{f^{-1}(0)}=C_{1}+C_{2}$ whose general member is smooth and whose base locus does not meet the singular locus of $C_{1}+C_{2}.$  

We may use the blowup $\alpha: \widetilde{X_d} \rightarrow X_d$ of $X_d$ at the base locus of $f$ to resolve indeterminacy and obtain a morphism $\widetilde{f}:\widetilde{X_d} \rightarrow \mathbb{P}^{1}$ whose fibers are the members of $f.$  Let $\beta: \widetilde{Y_d} \rightarrow \widetilde{X_d}$ be the blowup of $\widetilde{X_d}$ at the $m$ nodes of $C_1+C_2$ with exceptional divisors $E_{1}, \cdots ,E_{m},$ let $\widetilde{g}:=\widetilde{f}\circ \beta,$ and let $T=\{0\} \cup \{t \in \mathbb{P}^{1}: \widetilde{g}^{-1}(t) \textnormal{ is smooth}\}.$  In what follows, we will consider the family $\widetilde{g}_{T}:=\widetilde{g}|_{\widetilde{g}^{-1}(T)}:\widetilde{Y_d}_{T} \rightarrow T$ whose central fiber $C_{0}$ is $\widetilde{C_{1}} \cup \widetilde{C_{2}} \cup E_{1} \cup \cdots \cup E_{m},$ where $\widetilde{C_{1}}$ and $\widetilde{C_{2}}$ are the strict transforms under $\alpha \circ \beta$ of $C_{1}$ and $C_{2},$ respectively.

Define $M_{\widetilde{Y_d}_{T}}:=(\alpha \circ \beta|_{\widetilde{Y_d}_{T}})^{\ast}(\Omega^{1}_{\mathbb{P}^{d}}(1) \otimes \mathcal{O}_{X_d}).$  Then we have that 
\medskip
\begin{equation}
\label{rest1}
\wedge^{i}M_{\widetilde{Y_d}_{T}}|_{\widetilde{C_{1}}} \cong \wedge^{i}M_{C_{1}}, \hskip1cm \wedge^{i}M_{\widetilde{Y_d}_{T}}|_{\widetilde{C_{2}}} \cong \wedge^{i}M_{C_{2}}
\end{equation}
\begin{equation}
\label{rest2}
\wedge^{i}M_{\widetilde{Y_d}_{T}}|_{E_{j}} \cong \mathcal{O}_{E_{j}}^{\oplus \binom{d}{i}} \textnormal{ for }1 \leq j \leq m
\end{equation}
\begin{equation}
\label{rest3}
\wedge^{i}M_{\widetilde{Y_d}_{T}}|_{\widetilde{g}_{T}^{-1}(t)} \cong \wedge^{i}M_{\widetilde{f}^{-1}(t)} \textnormal{ for }t \in T-\{0\}
\end{equation}  
\medskip
By our assumptions on $C_{1}$ and $C_{2},$ for $j=1,2$ there exists a nonempty Zariski-open subset $\mathcal{U}_{j}$ of $\textnormal{Pic}^{g_{j}-1+r_{j}i}(C')$ such that for all $L_{j} \in \mathcal{U}$ we have the vanishings
\begin{equation}
\label{vann1}
H^{i}(C_{j},\wedge^{i}M_{C_{j}} \otimes L_{j})=0 
\end{equation}
Since the normalization of $C_{0}$ is the disjoint union $\widetilde{C'} \sqcup \widetilde{Q} \sqcup E_{1} \sqcup \cdots \sqcup E_{m}$, we have that for each $L_{1}$ and $L_{2}$ as above there exists a line bundle $\widetilde{L}$ on the singular curve $C_{0}$ such that 
\begin{equation}
\label{rest4}
\widetilde{L}|_{\widetilde{C_{1}}} \cong L_{1}, \hskip1cm \widetilde{L}|_{\widetilde{C_{2}}} \cong L_{2}, \hskip1cm \widetilde{L}|_{E_{i}} \cong \mathcal{O}_{E_{j}}(1) \textnormal{ for }1 \leq j \leq m 
\end{equation}
Such line bundles form a nonempty Zariski-open subset $\mathcal{V}$ of the Picard scheme $\textnormal{Pic}^{(g'-1+(r-1)i,i-1,1, \cdots 1)}(C_{0})$ parametrizing isomorphism classes of line bundles on $C_{0}$ whose restrictions to $\widetilde{C}_{1}$ and $\widetilde{C}_{2}$ have respective degrees $g_{1}-1+r_{1}i$ and $g_{2}-1+r_{2}i$ and whose restriction to $E_{j}$ has degree 1 for $1 \leq j \leq m.$  Note that the sum of the degrees of these restrictions is $g-1+ri.$

Fix a line bundle $\widetilde{L} \in \mathcal{V}.$  Passing to a finite base change $T' \rightarrow T$ if necessary, we see that there exists a line bundle $\mathcal{L}$ on $\widetilde{Y}_{T}$ whose restriction to each fiber of $\widetilde{g}_{T}$ has degree $g-1+ri$ and whose restriction to $C_{0}$ is isomorphic to $L.$  

Consider the exact sequence  
\begin{equation}
0 \rightarrow \bigoplus_{j=1}^{m}\mathcal{O}_{E_{j}}(-2) \rightarrow \mathcal{O}_{C_{0}} \rightarrow \mathcal{O}_{\widetilde{C_{1}}} \oplus \mathcal{O}_{\widetilde{C_{2}}} \rightarrow 0
\end{equation}
where the arrow into $\mathcal{O}_{C_{0}}$ is extension by zero and the arrow out of $\mathcal{O}_{C_{0}}$ is the direct sum of the restriction maps $\mathcal{O}_{C_{0}} \rightarrow \mathcal{O}_{\widetilde{C_{1}}}$ and $\mathcal{O}_{C_{0}} \rightarrow \mathcal{O}_{\widetilde{C_{2}}}.$  If we twist this by $M_{\widetilde{Y}_{T}} \otimes \mathcal{L}$ and take cohomology, it follows from (\ref{rest1}), (\ref{rest2}), (\ref{vann1}) and (\ref{rest3}) that $H^{0}(C_{0},(M_{\widetilde{Y}_{T}} \otimes \mathcal{L})|_{C_{0}})=0.$  By semicontinuity, we then have that for $H^{0}(\widetilde{g}^{-1}(t),M_{W,\widetilde{g}^{-1}(t)} \otimes \mathcal{L}|_{\widetilde{g}^{-1}(t)})=0$ for general $t \in T-\{0\}.$       
\end{proof}

\textit{Proof of Theorem \ref{mrchyp}}:  The implication $(i) \Rightarrow (ii)$ follows from Propositions \ref{thetaopen} and \ref{mrcsum}.  On the other hand, $(ii) \Rightarrow (i)$ is an immediate consequence of Theorem \ref{main} and Proposition \ref{fmpmrc}. \hfill \qedsymbol

\medskip

Combining Theorem \ref{mrchyp} with Lemma \ref{ratnorm} yields the following consequence:

\begin{cor}
\label{ratgen}
If $\mathfrak{Ulr}(X_d)$ is generated by classes of rational normal curves of degree $d,$ then $D \in {\rm Pic}(X_d)$ belongs to $\mathfrak{Ulr}(X_d)$ if and only if MRC holds for the general smooth member of $|D|.$ \hfill \qedsymbol
\end{cor}

\begin{rem}
\label{degree7}
As mentioned in the Introduction, the hypothesis of Corollary \ref{ratgen} is satisfied for the smooth cubic surface $X_3.$  On the other hand, while the degree-7 del Pezzo surface $X_7$ contains two rational normal curve classes $Q_1$ and $Q_2$, they do not generate $\mathfrak{Ulr}(X_7).$  Proposition \ref{hplusrnc} implies that there exist rank-2 Ulrich bundles with first Chern classes $H+Q_1$ and $H+Q_2,$ and neither one of these classes is equal to $Q_1+Q_2.$
\end{rem}

\end{document}